\documentclass[11pt]{amsart}
\usepackage[utf8]{inputenc}
\usepackage{fullpage,url,amssymb,enumitem,colonequals, mathrsfs,comment}
\usepackage{microtype}
\usepackage{hyperref}
\usepackage{xcolor}
\usepackage{tikz}
    \usetikzlibrary{calc}
\usepackage{tikzit}
\usepackage[pdftex]{}
\usepackage{mathabx}
\usepackage{cleveref}
\usepackage{array}
\usepackage{tabularx}
\usepackage{mathtools}

\newcommand{\Aff}{\mathbb{A}}
\newcommand{\C}{\mathbb{C}}

\newcommand{\F}{\mathbb{F}}
\newcommand{\G}{\mathbb{G}}

\newcommand{\Q}{\mathbb{Q}}
\newcommand{\R}{\mathbb{R}}
\newcommand{\Z}{\mathbb{Z}}
\newcommand{\Qbar}{{\overline{\Q}}}

\newcommand{\calD}{\mathcal{D}}

\newcommand{\calI}{\mathcal{I}}

\newcommand{\calP}{\mathcal{P}}

\newcommand{\calS}{\mathcal{S}}

\newcommand{\BB}{{\mathscr{B}}}
\newcommand{\EE}{{\mathscr{E}}}
\renewcommand{\AA}{{\mathscr{A}}}

\newcommand{\KK}{\mathscr{K}}
\newcommand{\MM}{\mathscr{M}}

\newcommand{\UU}{\mathscr{U}}

\newcommand{\bfj}{\boldsymbol{j}}

\newcommand{\bfm}{\boldsymbol{m}}

\newcommand{\bfC}{\boldsymbol{C}}
\newcommand{\bfD}{\boldsymbol{D}}

\newcommand{\bfG}{\boldsymbol{G}}

\newcommand{\bfM}{\boldsymbol{M}}

\DeclareMathOperator{\Gal}{Gal}

\DeclareMathOperator{\Prob}{Prob}

\newcommand{\tors}{{\operatorname{tors}}}

\newcommand{\Leg}{\operatorname{Leg}}

\newcommand{\isom}{\simeq}

\newcommand{\tensor}{\otimes} 
\newcommand{\Union}{\bigcup} 

\newcommand{\MMgeneric}{%
  \hyperlink{special*}{%
    \tikz[baseline=-0.55ex,
          x=0.45ex,y=0.45ex,
          line cap=round,line join=round]{
      \coordinate (A) at (0.7,3.0);
      \coordinate (B) at (0,0);
      \coordinate (C) at (3.4,-0.2);
      \coordinate (AB) at ($(A)!0.5!(B)$);
      \coordinate (AC) at ($(A)!0.5!(C)$);
      \coordinate (BC) at ($(B)!0.5!(C)$);
      \draw[line width=.12ex] (A)--(B)--(C)--cycle;
      \draw[line width=.08ex]
        (A)--(BC)
        (B)--(AC)
        (C)--(AB);
    }%
  }%
}

\newtheorem{theorem}{Theorem}[section]
\newtheorem*{theorem*}{Theorem}
\newtheorem{lemma}[theorem]{Lemma}
\newtheorem{corollary}[theorem]{Corollary}
\newtheorem{proposition}[theorem]{Proposition}

\theoremstyle{definition}
\newtheorem{definition}[theorem]{Definition}
\newtheorem{question}[theorem]{Question}

\theoremstyle{remark}
\newtheorem{remark}[theorem]{Remark}

\usepackage{microtype}
\title{Generic Manin--Mumford}

\author{Lior Bary-Soroker}
\address{Lior Bary-Soroker, School of Mathematical Sciences, Tel Aviv University, Tel Aviv 6997801, Israel}
\email{barylior@tauex.tau.ac.il}
\urladdr{\url{https://www.math.tau.ac.il/~barylior/}}

\author{Borys Kadets}
\address{Borys Kadets, Einstein Institute of Mathematics\\
Hebrew University of Jerusalem}
\email{kadets.math@gmail.com}
\urladdr{\url{http://bkadets.github.io}}

\date{\today}

\begin{document}
\begin{abstract}
    Given a collection of algebraic numbers $\calS\subset \Qbar$ we study the varieties $V$ in $\Aff^m_\C$ 
    such that $V(\C)\cap\calS^m$ is Zariski-dense in $V$. We show that for many classical families of algebraic numbers $\calS$---such as the family of roots of generalized Laguerre polynomials $L_n^{(\alpha)}(x)$, for a finite collection of $\alpha\in \Q$---an unlikely intersections theorem holds. For example, in the case $m=2$, we prove that an irreducible curve in $\Aff^2_\C$ has infinitely many points from $\calS^2$ if and only if it is of the form $x_1=x_2,$ or $x_1=s$, or $x_2=s$ for a fixed $s \in \calS$.  This is an analogue of the classical theorems of Ihara, Serre, and Tate, treating  the case of $\calS$ consisting of the roots of unity, and of the Manin--Mumford conjecture.
    We also show that a similar result holds almost surely for roots of a collection of random polynomials of growing degree and bounded height. The proofs rely on a uniform Galois-theoretic criterion ensuring the unlikely intersection property. 
\end{abstract}
\maketitle

\section{Introduction}
The unlikely intersection principle asserts that an algebraic variety
containing many \emph{special} points must itself be \emph{special}.
An early example is the following theorem of Ihara, Serre, and Tate (recorded by Lang) \cite{Lang1965}:
\begin{quote}
    \it
    Let $F(X,Y)\in\C[X,Y]$ be irreducible. If the equation
    $F(X,Y)=0$ has infinitely many solutions in roots of unity, then, up to multiplication by a nonzero constant, 
    \[
        F= X^aY^b-\zeta
        \qquad\text{or}\qquad
        F=X^a- \zeta Y^b,
    \]
    $\zeta\in\C$ is a root of unity 
    $a,b\ge 0$ are integers.
\end{quote}
Here the special points are the torsion points of $\G_m^2$, and the
special curves are translates of one-dimensional subtori by torsion
points.

The Manin--Mumford conjecture, proved by Raynaud \cite{Raynaud1983}, predicts that a subvariety of an abelian
variety containing a Zariski dense set of torsion points must be a
finite union of torsion translates of abelian subvarieties.  Laurent \cite{Laurent1984}
established the analogous statement for algebraic tori, while the Andr\'e--Oort conjecture, resolved in \cite{Pila2011,Tsimerman2018,KlinglerUllmoYafaev2022} using the so-called Pila--Zanier strategy \cite{pila2008rational}, provides a further analogue for special points
on Shimura varieties; see the survey paper \cite{Tsimerman2024}.

Motivated by these results, we consider analogous questions for other
collections of algebraic numbers. Given an infinite set
\(\calS\subseteq \overline{\Q}\), whose elements we regard as
\emph{special points}, we ask which subvarieties of \(\Aff^n\) can
contain a Zariski dense subset of \(\calS^n\).

Without additional assumptions on \(\calS\), the answer can be arbitrarily complicated. For example, if \(\calS=\overline{\Q}\), then
every subvariety of \(\Aff^n\) contains a Zariski dense subset of \(\calS^n\). 

\begin{definition} \noindent
\begin{enumerate}
    \item A variety in $\Aff^n_\C$ is called \emph{$\calS$-special} if it is a linear subspace cut out by a finite, possibly empty, collection of equations of the form $x_i=x_j$, $i\neq j$ or $x_i=s$, $s\in\calS$. Similarly, we say that a variety is \emph{$\calS^{\pm}$-special} if in addition one allows equations $x_i=-x_j$, and \emph{$\calS^{\tors}$-special} if one allows $x_1^{i_1}x_2^{i_2}\dots x_n^{i_n}=\zeta$, $i_j \in \Z$, where $\zeta$ is a root of unity.
    \item 
    We say that $\calS\subseteq \overline{\Q}$ is \emph{special} (respectively, \textit{signed-special}, and \textit{torsion-special}) if the Zariski closure of any $B\subseteq \calS^n$ in $\Aff^n_{\C}$ is a finite union of $\calS$-special subvarieties (respectively, $\calS^{\pm}, \calS^{\tors}$-special). 
\end{enumerate}
\end{definition}

The following theorem, communicated to us by Binyamini, Kiro, and Pila \cite{BinyaminiKiroPila}
shows that the roots of Laguerre polynomials satisfy a Manin--Mumford
type theorem.
Recall the $n$-th Laguerre polynomial is 
\[
    P_n (x) = \sum_{k=0}^n\binom{n}{k}\frac{(-1)^k}{k!}x^k.
\]
and it is a solution to the Laguerre differential equation 
\[
    xy''+(1-x)y'+ny=0.
\]
\begin{theorem*}[Binyamini, Kiro, and Pila]
    The set 
    \[
        \calS = \{\alpha\in \C: P_n(\alpha)=0 \text{ for some }n\geq 1\}
    \]
    of roots of Laguerre polynomials 
    is special.
\end{theorem*}

The proof given in \cite{BinyaminiKiroPila} uses techniques of o-minimality, functional transcendence, and differential elimination.
We give a Galois-theoretic proof of this theorem, that also applies to various other classical families of polynomials as well as to random polynomials of bounded height. Our intention is to show that most collections of algebraic numbers of growing degree one encounters in nature are special (and so the term special is a bit of a misnomer).    

\begin{theorem}\label{thm:special-polynomials-intro}
    Suppose $\calS \subset \Qbar$ is the set of roots of any of the following families of polynomials:

    \begin{enumerate}
        \item Truncated exponentials: $e_n(x)=1+x+x^2/2+ \dots + x^n/n!$;
        \item Selmer trinomials: $p_n(x)=x^n-x-1$;
        \item Bessel: $y_n(x)=\sum_{i=0}^n \frac{(n+i)!}{2^i (n-i)! i!}x^i$.
        \item Generalized Laguerre:
        $
        L_n^{(\alpha_j)}(-x)
        =
        \sum_{i=0}^{n}
        \binom{n+\alpha_j}{n-i}
        \frac{x^i}{i!}
        $, $\alpha_1,\ldots, \alpha_r \in \Q\smallsetminus \Z_{<0}$ and $n\geq 1$.
    \end{enumerate}

Then $\calS$ is special. 
   If $\calS$ is the set of roots of any of the following families of polynomials, then $\calS$ is signed-special.

     \begin{enumerate}
         \setcounter{enumi}{4} 
      \item Truncated cosine: $\cos_n(x)=1-x^2/2!+\dots +(-1)^nx^{2n}/(2n)!$
        \item Hermite: $H_n(x)=(-1)^ne^{x^2}\frac{d^n}{dx^n} e^{-x^2/2}$;
        \item Legendre: $\mathrm{Leg}_p(x)=
        \frac{1}{2^p p!} \frac{d^p}{dx^p}(x^2-1)^p$, $p$ prime, \emph{assuming the Hardy--Littlewood conjecture \cite[Conjecture D]{hardy1923some}}.
    \end{enumerate}
\end{theorem}
\begin{remark}
    Arithmetic of some classical polynomial sequences is quite difficult to understand. A particular notorious case is that of Legendre polynomials $\mathrm{Leg}_n(x)$ for all $n$, where not only one cannot prove the irreducibility, but even the easier question of whether $\Leg_n$ and $\Leg_m$ might share nontrivial roots is open; these questions are known as Stieltjes conjecture(s), as they appear in Stieltjes's letter to Hermite \cite{stieltjes-hermite-letter}. See also \cite{cullinan2014galois} for an up-to-date summary of arithmetic of Legendre polynomials. 
\end{remark}

We expect that the property of being special holds for much larger sets $\calS$ of algebraic numbers of growing degree. For example, we propose the following question.

\begin{question}
Fix $H\geq 1$ and let $\calS_H$ be the set of roots of all monic polynomials 
in $\Z[x]$
with coefficients bounded by $H$. Is $\calS_H$ torsion-special?    
\end{question}

We prove that the answer is yes for a randomly sampled family of polynomials. 

\begin{theorem}\label{thm:RandomMM}
    There exists $H_0\geq 2$ such that for any set $I\subseteq \Z$ of at least $H_0$ consecutive integers the following holds. Let $f_1,f_2,\ldots$ be a sequence of independent random polynomials such that $f_n$ is chosen uniformly at random from the set of monic polynomials of degree $n$ with coefficients in $I$ and with nonzero free coefficient. Let $\calS \subseteq \overline{\Q}$ be the set of all roots of all $f_n$.
    Then, $\calS$ is special almost surely. 
\end{theorem}

The proofs of both the deterministic and probabilistic results are based
on a common Galois-theoretic criterion. Roughly speaking, it asserts
that a set of algebraic numbers is special
whenever its elements have growing degree, large Galois groups, and
almost disjoint splitting fields.

In order to be more precise, we recall that for $\alpha\in \overline{\Q}$, one sets $\deg \alpha  \coloneqq  [\Q(\alpha):\Q]$. We let $L_{\alpha}$ be the splitting field of $\Q(\alpha)/\Q$ and 
\[
    G_{\alpha}\coloneqq \Gal(L_{\alpha}/\Q),
\]
viewed as a permutation group induced from the Galois action on the conjugates of $\alpha$. 

\begin{definition}\label{def:Galois-typical}
\hypertarget{special*}{}
    A subset $\calS \subset \Qbar$ is said to satisfy \MMgeneric{} if the following conditions hold.
    \begin{enumerate}
        \item {\it Growing degree:} For every $N$ there are only finitely many $\alpha\in\calS$ with $\deg \alpha < N$.
        \item {\it Large Galois group:} $G_{\alpha}\in \{S_{\deg \alpha},A_{\deg \alpha}\}$ for all but finitely many $\alpha \in \calS$.
        \item {\it Disjointness:} For all but finitely many pairs $(\alpha, \beta) \in \calS^2$ either $\alpha$ and $\beta$ are  conjugate or $[L_{\alpha}\cap L_{\beta} :\Q]\leq 2$.
    \end{enumerate}
\end{definition}

\begin{theorem}[Generic Manin--Mumford]\label{thm:GMM-intro}
    If $\calS\subseteq \overline{\Q}$ satisfies \MMgeneric{} then $\calS$ is special. 
\end{theorem}

All the proofs of the theorems above are reduced to showing condition \MMgeneric{}.

\begin{remark}
  It is natural to wonder whether other unlikely intersections statements, beyond the Manin--Mumford theorem, hold for points with generic coordinates. One natural candidate would be a Mordell--Lang type theorem for the subgroup generated by points in $\calS^n$. However, it is easy to produce counterexamples to this claim. For example, consider the curve $X \subset \G_m^2$ given by the equation $x+2y+6=0$. This curve is not a translate of a subgroup of $\G_m^2$, and thus one might expect that at most finitely many points on $X$ belong to the subgroup generated by $\calS^2$ (for $\calS$ satisfying \MMgeneric{}). We now show that this is not the case.
  
  For every prime $p>5$, let $\alpha_{p}$ be the root of the polynomial $f(t)=t^{p+3}+2t^{p}+6$. This polynomial is Eisenstein at $2$ and thus irreducible. Moreover, from the $3$-adic Newton polygon of $f$, its Galois group contains a $p$-cycle (coming from the tame inertia action). By Jordan's theorem, the Galois group of $\Q(\alpha_p)/\Q$ contains $A_{p+3}$. Thus the sequence $\alpha_p$ satisfies \MMgeneric{}. However, the points $(\alpha_p^{p+3}, \alpha_p^p)$ are dense in $X$.     
\end{remark}

The paper is organized as follows. In Section~\ref{sec:gMM} we prove Theorem~\ref{thm:GMM-intro}. Then, in Sections~\ref{sec:classical} and \ref{sec:Laguerre-Manin-Mumford}, we establish \MMgeneric{} for the zeros of the classical families of polynomials appearing in Theorem~\ref{thm:special-polynomials-intro}, hence proving it using  Theorem~\ref{thm:GMM-intro}. Finally, in Section~\ref{sec:RMM}, we establish \MMgeneric{} almost surely, hence proving Theorem~\ref{thm:RandomMM}.

\section*{Acknowledgements}

L.B.-S. was supported by the Israel Science Foundation (ISF), Grant No.~366/23.

B.K. was supported by the Israel Science Foundation (ISF), Grant No.~2112/26.

\section{Proof of Theorem \ref{thm:GMM-intro}}\label{sec:gMM}

Let $\calS \subseteq \overline{\Q}$ be a subset satisfying \MMgeneric{}. We want to prove that $\calS$ is special. We start by showing elements in $\calS$ have long relative Galois orbits:

\begin{lemma}\label{Galois-disjoint}
    Suppose \MMgeneric{} holds. Then there is an $N$ such that for every $n\geqslant 1$ and for all strings $(\alpha_1, \dots, \alpha_n) \in \calS^n$ with $\deg(\alpha_1) \geqslant \dots \geqslant \deg (\alpha_n) \geqslant N$ and $\alpha_i \neq \alpha_j$ we have 
    \[[\Q(\alpha_1, \dots, \alpha_n)/\Q(\alpha_2, \dots, \alpha_n)] \geqslant \deg(\alpha_1)-n.\]

\end{lemma}

\begin{proof}
Let $I_1,\ldots, I_s$ be the partition of $\{\alpha_1,\ldots, \alpha_n\}$ into sets of conjugates with $\alpha_1\in I_1$. For each $j=1,\ldots,s$, let $d_j=\deg \alpha$ and $L_j=L_{\alpha}$, for some $\alpha\in I_j$ (this is clearly independent of the choice of $\alpha$). Let $L=L_1\cdots L_s$ and let $E\subseteq L$ be the compositum of all quadratic subfields of $L$. 

The assertion is trivial if $n\geqslant  \deg\alpha_1+1$ hence we always assume $n<\deg\alpha_1+1$. 
We may assume that $N\geq 5$ is  sufficiently large so that $L_{j_1}\cap L_{j_2}\subseteq E$ for all $1\leq j_1<j_2\leq s$ and $\Gal(L_j/\Q) \in \{S_{d_j} ,A_{d_j}\}$. Since $d_j\geq 5$, $A_{d_j}$ is both simple non-abelian and the unique subgroup of index $2$ in $S_{d_j}$. Hence, $\Gal(L_jE/E)=A_{d_j}$. Ribet's lemma and pairwise disjointness of $L_jE/E$ then gives  
\[
    \Gal(L/E) \cong \prod_{j=1}^s \Gal(L_jE/E) \cong \prod_{j=1}^s A_{d_j},
\]
as we now explain. The proof is by induction on $s$ with the base case $s=2$ given by the assumption. Let $L'$ be the compositum $EL_2\dots L_s$; by the induction hypothesis $\Gal(L'/E)=\prod_{j=1}^{s-1}A_{d_j}.$ Since $A_{d_1}$ is a simple group, if $\Gal(L/E)\neq A_{d_1}\times \Gal(L'/E)$, then $L_1E$ is a subfield of $L'$. Let $H\subset \prod_{j=1}^{s-1}A_{d_j}$ be the subgroup corresponding to $L_1E$. By pairwise disjointness, $H$ surjects onto each of the factors $A_{d_j}$, and hence, by Ribet's lemma, $H=\Gal(L'/E)$, giving a contradiction.

In particular, if we set $L'=EL_{2}\cdots L_{s}$, then $L=L_1L'$ and $\Gal(L/L')=A_{d_1}$. If we further write $I_1 = \{\alpha_1,\alpha'_2,\ldots, \alpha'_m\}$, then we conclude 
\[
    [\Q(\alpha_1,\ldots, \alpha_n):\Q] \geq [L'(\alpha_1,\alpha'_2,\ldots, \alpha'_m):L'(\alpha'_2,\ldots, \alpha'_m)] =[A_{d_1-m+1}:A_{d_1-m}] = d_1-m\geq d_1-n. 
\]
Here, $A_{d_1-m}$ and $A_{d_1-m+1}$ are the respective stabilizers of $I_1$ and $I_{1}\smallsetminus \{\alpha_1\}$ in $\Gal(L/L').$
\end{proof}

    We prove that the Zariski closure $\bar{B}$ of  $B\subseteq \calS^n$ is a finite union of $\calS$-special varieties  by induction on $n$. We begin with some simple reductions.

    It suffices to show that $B$ is contained in a finite union of special hyperplanes. Indeed,  restricting to each of the hyperplanes, we may then drop one of the variables, and apply the induction hypothesis. 
    
    It suffices to assume that each of the coordinate projections from $\bar{B}$ is dominant (and in particular $\bar{B}$ is a hypersurface). Otherwise, we may  use that the preimage of a special subvariety under a coordinate projection is special. 
    
    Moreover, we can assume that $\bar{B}$ is defined over some finite extension $K$ of $\Q$. 

    Consider the closed special set $Z$ consisting of the diagonals $x_i=x_j$ and all hyperplanes $x_i=\alpha$ for $\alpha \in \calS$ of degree less than $\max([K:\Q](\deg(\bar{B})+1) + n, N)$. Let $B'=B\smallsetminus(B\cap Z)$. Note that, by the above reductions, we can assume that $\bar{B}'$ is a hypersurface and each of the coordinate projections is dominant. Moreover, since $\bar{B}' \subset \bar{B}$, we have that $\deg \bar{B}'\leqslant \deg \bar{B}$; so without loss of generality we may replace $B$ by $B'$.
    
        If $\bar{B}$ is not a union of special hyperplanes, then it does not change if we remove from $B$ intersections with the diagonals $x_i=x_j$. Since $\bar{B}$ is not contained in the hyperplanes $x_i=\alpha$, we can remove finitely many points from $B$ and assume that all coordinates of all points of $B$ have degree larger than $\max([K:\Q](\deg(\bar{B})+1) + n, N)$, where $N$ is the constant of \Cref{Galois-disjoint}. Finally, by decomposing $B$ into a union of $n!$ sets, we can assume that all points $(\alpha_1, \dots, \alpha_n) \in B$ satisfy \[\deg(\alpha_1) \geqslant \dots \geqslant \deg (\alpha_n)\geqslant \max \left([K:\Q](\deg(\bar{B})+1)+n, N\right).\]

    Consider a point $\alpha=(\alpha_1, \dots, \alpha_n) \in B$. By \Cref{Galois-disjoint}, the $\Gal_\Q$ orbit of $\alpha$ contains at least $\deg (\alpha_1)-n$ points on the line $(t, \alpha_2, \dots, \alpha_n)$. Therefore the $\Gal_K$ orbit of $\alpha$ contains at least $(\deg (\alpha_1)-n)/[K:\Q] \geqslant \deg (\bar{B})+1$ points on that line. Since $\bar{B}$ is defined over $K$, all of the orbit has to belong to $\bar{B}$. But a line can intersect $\bar{B}$ in $\deg(\bar{B})+1$ points only if it belongs to $\bar{B}$ entirely. Since this is true for every point of $B$, the hypersurface $\bar{B}$ is also the Zariski closure of $\pi_1^{-1}(\pi_1(B))$, where $\pi_1$ denotes the coordinate projection. This means that the fibers of the projection $\pi_1: \bar{B} \to \Aff^{n-1}$ are not zero-dimensional above all the points of $\pi_1(B)$, and thus $\pi_1$ is not dominant. This is a contradiction.    \qed

\section{Classical polynomial sequences}\label{sec:classical}
Condition \MMgeneric{} is particularly easy to verify for a set $\calS$ which is a union of roots of a system of polynomials $P_n$ of growing degree. This leads to a Manin--Mumford type statement for many classical families of polynomials. In this section we prove all cases of Theorem \ref{thm:special-polynomials-intro} except for the family of generalized Laguerre polynomials, which requires more work.

\begin{theorem}
    Suppose $\calS \subset \Qbar$ is the collection of (all) roots of any one of the following families of polynomials:

    \begin{enumerate}
        \item Truncated exponentials: $e_n(x)=1+x+x^2/2+ \dots + x^n/n!$;
        \item Selmer trinomials: $p_n(x)=x^n-x-1$;
      
        \item Bessel polynomials: $y_n(x)=\sum_{i=0}^n \frac{(n+i)!}{2^i (n-i)! i!}x^i$.
    \end{enumerate}

    Then the set $\calS$ satisfies \MMgeneric{}, and so for any subset $B \subset \calS^n \subset \Aff^n(\Qbar)$, the Zariski closure $\bar{B}$ of $B$ in $\Aff^n_{\C}$ is a finite union of special subvarieties.
\end{theorem}
\begin{proof}
    Note that if $P_n(x)$ is a sequence of degree $n$ polynomials for which, for large $n$, the polynomials $P_n(x)$ have Galois groups containing $A_n$, then the condition \MMgeneric{} is automatically satisfied. The disjointness property of \MMgeneric{} follows from group theory: since $A_n$ is a simple group for $n\geqslant 5$, Galois extensions with groups containing $A_n$ and $A_m$ respectively intersect by at most a quadratic extension whenever $n>m>5$. Therefore we just need to verify that the listed polynomials have large Galois groups.
    
    For the truncated exponentials $e_n(x)$ this is the classical theorem of Schur \cite{schur1930gleichungen}; see also Coleman's proof \cite{coleman1987galois}. For Selmer trinomials $p_n(x)$, the irreducibility was proved by Selmer \cite{selmer1956irreducibility} (hence the name), and the Galois group statement was proved by Nart and Vila \cite{nart1979equations}. For Bessel polynomials $y_n(x)$, eventual irreducibility was proven relatively recently by Filaseta \cite{filaseta1995irreducibility}; previously\footnote{the first edition of the cited book was published in 1978} Grosswald proved that irreducibility implies that the Galois group contains $A_n$ \cite[Chapter 12]{grosswald2006bessel}.   
\end{proof}

Some classical families of polynomials are odd (or even), and so the corresponding families $\calS$ are not special, but only signed-special. Nevertheless, we can reduce these cases to Theorem \ref{thm:GMM-intro}.

\begin{theorem}
    Suppose $\calS\subset \Qbar$ is the collection of roots of one of the following families of polynomials:
    \begin{enumerate}
        \item Truncated cosine: $\cos_n(x)=1-x^2/2!+\dots +(-1)^nx^{2n}/(2n)!$;
        \item Hermite polynomials $H_n(x)=(-1)^ne^{x^2}\frac{d^n}{dx^n} e^{-x^2/2}$;
        \item  Legendre polynomials: $\mathrm{Leg}_{p+1}(x)=\frac{1}{2^{p+1}(p+1)!} \frac{d^{p+1}}{dx^{p+1}}(x^2-1)^{p+1}$ for $p$ prime \emph{assuming the Hardy--Littlewood conjecture \cite[Conjecture D]{hardy1923some}}.
    \end{enumerate}
Then for any subset $B\subset \calS^n\subset \Aff^n(\Qbar)$ the Zariski closure $\bar{B}$ of $B$ is a finite union of affine linear subspaces $\Lambda$ each of which is, in turn, a finite intersection of hyperplanes of the form $x_i=x_j$ or $x_i=-x_j$ or $x_i=s$ for $s\in \calS$.
\end{theorem}
\begin{proof}
    Consider the map $\phi_n:
    \Aff^n_\Q \to \Aff^n_\Q$ which squares all coordinates. Then the Galois groups of all but finitely many elements of $\hat{S}\coloneqq \phi_1(\calS)$ contain the alternating group; for $\cos_n$ this is a theorem of Shokri, Shafaff, and Taleb \cite{shokri2019galois}, for Hermite polynomials this is due to Schur \cite{schur1931affektlose}, and for Legendre polynomials, conditionally on the Hardy--Littlewood conjecture, this follows from \cite[Theorem 1.8]{cullinan2014galois} and \cite{holt1913irreducibility}. 
    
    Therefore $\hat{S}$ satisfies \MMgeneric{}, and so, by \Cref{thm:GMM-intro} the Zariski closure of the set $\phi_n(B)$ is a finite union of special linear subspaces. Since a special affine linear subspace $V$ is cut out by equations of the form $x_i=x_j$ and $x_i=s^2$, for $s \in \calS$, the preimage $\phi_n^{-1}(V)$ of such a subspace is a finite union of affine linear subspaces cut out by equations $x_i=\pm x_j$ and $x_i=s$, as claimed. \end{proof}

\section{Laguerre--Manin--Mumford}\label{sec:Laguerre-Manin-Mumford}
In this section we will prove the Laguerre--Manin--Mumford theorem. We first summarize some basic properties of Laguerre polynomials in the following proposition.

\begin{proposition}\label{prop:basics-laguerre}
    The generalized Laguerre polynomials $L_n^\alpha \in \Q[x]$ are a family of polynomials defined for every $\alpha \in \Q$ and satisfying the following properties:

    \begin{enumerate}\label{thm:Laguerre-Galois}
        \item We have the formula \[L_n^{\alpha}(x)=\sum_{i=0}^n \frac{(n+\alpha)(n-1+\alpha)\dots (i+1+\alpha)}{i!(n-i)!}(-x)^i;\]
        \item If $\alpha$ is not a (strictly) negative integer, then for all but finitely many $n$ the polynomial $L_n^{\alpha}(x)$ is irreducible and has Galois group which contains $A_n$;
        \item For an integer $k$, $L_n^{-k}(x)=x^k L^{k}_{n-k}(x)$;
        \item The discriminant of $L_n^\alpha$ is 
        \[\prod_{i=2}^n i^i (\alpha+i)^{i-1}.\]
    \end{enumerate}
\end{proposition}
\begin{proof}
    The large Galois group statement was proved by Hajir \cite{hajir2005galois} building on the work of Filaseta and Lam \cite{Filaseta2002}. The remaining statements are classical; see for example the introduction to \cite{hajir2004algebraic} for these (and other) basic properties.
\end{proof}

To apply \Cref{thm:GMM-intro} we need to establish Galois disjointness for Laguerre polynomials. This is the content of the following lemma.

\begin{lemma}\label{lem:Laguerre-disjoin}
    For every $\alpha, \beta \in \Q\setminus \Z_{<0}$ there exists $N=N(\alpha,\beta)>0$ such that for every $n,m\geq N$, if  the splitting fields of the polynomials $L_{n}^\alpha$ and $L_{m}^\beta$ intersect by more than a quadratic extension, then $n=m$ and $\alpha=\beta$.
\end{lemma} 

\begin{proof}
    Suppose $n\geqslant m$, and let $F_{n,\alpha}=\Q[x]/(L_{n}^\alpha(x))$ be the field generated by a root of $L_{n}^\alpha$, $K_{n,\alpha}$ the splitting field of $L_n^{\alpha}$, and similarly define $F_{n,\beta}$ and $K_{n,\beta}$. By \Cref{thm:Laguerre-Galois}, we can choose $N$ sufficiently large so that the Galois group of $K_{n,\alpha}$ is $S_n$ or $A_n$ and $n\geq 5$. If $K_{n,\alpha}$ intersects another Galois extension $L/\Q$ by more than a quadratic subfield, then $K_{n, \alpha} \subset L$ by the simplicity of the group $A_n$. In particular,  if $[K_{n,\alpha} \cap K_{m, \beta}:\Q]>2$, then $n=m$ and $K_{n,\alpha}=K_{n,\beta}$. Note that this immediately implies that  $F_{n,\alpha}\cong F_{n,\beta}$. Hence, it remains to prove that if $n$ is sufficiently large and $\alpha>\beta$, then $K_{n, \alpha}\neq K_{n, \beta}$. 
    
    Assume now, for the sake of contradiction, that $F_{n,\alpha}\cong F_{n,\beta}$ and  $K_{n, \alpha}=K_{n, \beta}$ for infinitely many $n$. 
    
    First we compare the discriminants as elements of $\Q^\times/\Q^{\times 2}$ to conclude that $\alpha=\beta + 2$: Using  \Cref{thm:Laguerre-Galois}(4) we get that the equation
    \begin{equation}
    \label{eq:disc_square_eq}    
    (\alpha+2)(\alpha+4)\dots (\alpha+ 2\lfloor n/2 \rfloor)(\beta+2)\dots (\beta+2 \lfloor n/2\rfloor) =y^2
    \end{equation}
    has infinitely many solutions $(n, y)$. The treatment of this equation depends on whether $\alpha-\beta$ is an even integer, and we separate the discussion accordingly. 

    \medskip
    We claim that $\alpha-\beta\in 2\Z$. Otherwise,  let $\alpha'=\alpha/2$, $\beta'=\beta/2$ and $\alpha'-\beta'\not\in \Z$. Write $\alpha'=a/b$ and $\beta'=c/d$ as simple fractions.  Suppose without loss of generality that $b \geqslant d$. Then for infinitely many $n$ the expression 
    \begin{equation}
    \label{eq:alpha'beta'prod}
    (\alpha'+1)(\alpha'+2)\dots(\alpha'+\lfloor n/2\rfloor)(\beta'+1)(\beta'+2)\dots(\beta'+\lfloor n/2\rfloor)
    \end{equation} is a square up to a power of $2$. 
    
    By the prime number theorem in arithmetic progressions, for any $\epsilon>0$ and for all sufficiently large $n$ there is an integer $\lfloor n/2 \rfloor>j>n/2 (1-\epsilon)$ such that $a+jb=p$ is prime, that is, $\alpha'+j=\frac{p}{b}$. Clearly $p\nmid bd$, hence it divides the numerator of other factor in  \eqref{eq:alpha'beta'prod}. Since $p\geqslant (n/2)b(1-\epsilon)$, none of the numerators of  $\alpha'+j'$, $j'\neq j$ is divisible by $p$ (otherwise, $p \mid b(j-j')$, but $p\nmid b$  and $p>n/2>j-j'$.) 

    Therefore there exists $i$ such that $p$ divides the numerators of $\beta'+i=\frac{c+di}{d}$. But $p \geqslant (n/2)d (1-\epsilon)$, so $p = c+d_i$. In particular, we get that 
    \[
        \frac{n}{2}b(1-\epsilon) \leq p\leq c+\frac{n}{2}d,
    \]
    for all sufficiently large $n$. 
    This implies that $b=d$ (otherwise, take $\epsilon<\frac{b}{d}-1$). So $c+ib=a+jb$, hence $\alpha'-\beta'=(a-c)/b=(i-j)\in \Z$, contradicting the assumption. Hence we deduce that $\alpha=\beta +2m$ for some positive integer $m$, as claimed. 

    \medskip
    Next we claim that $m=1$, i.e., $\alpha=\beta+2$. 
    Indeed, plugging $\alpha= \beta+2m$ in \eqref{eq:disc_square_eq} and absorbing all repeated factors by a change of  variable $y\mapsto cy$, we get that the simplified equation 
    \[
    (\beta+2)\dots (\beta+2m) 
    (\beta+2+2\lfloor n/2 \rfloor)\dots (\beta+2m+ 2\lfloor n/2 \rfloor)=y^2
    \]
    has infinitely many solutions $(n,y)$. 
    
    Setting $x=2\lfloor n/2\rfloor$ and multiplying by a common denominator, we see that the hyperelliptic equation
    \[C_1y^2=C_2(x+c_3)(x+(c_3+2))\dots (x+(c_3+2m))\]
    has infinitely many integer solutions $(x,y)$ (for some constants $C_1, C_2, C_3$ which depend on $\alpha$ and $\beta$). By Siegel's theorem, $m=1$, as claimed.

    \medskip
    To this end, we have  $\alpha=\beta+2$. Discriminant considerations are not sufficient to get a contradiction, and we analyze this by comparing Newton polygons at carefully chosen primes and using the isomorphism $F_{n,\alpha}\cong F_{n,\beta}$. 

    \medskip
    We claim that $\alpha$ (and therefore $\beta$) is an integer. Otherwise, we write $\alpha=a/b$  and choose a prime number $p$ of the form $mb+a$ for $n/2<m<n-2$. Then the $p$-adic valuation of the coefficients of $x^i$ in $L_n^\alpha$ is clear from the formula: it is $1$ for $i<m$  and zero for $i\geqslant m$. Therefore the polynomial $L_n^\alpha$ has a root of valuation $1/m$. Since $m>n/2$, this means that the \'etale algebra $F_{n, \alpha} \tensor \Q_p$ has a component which is a totally ramified field extension of $\Q_p$ of degree $m$. Applying the same analysis to $L_n^\beta$, shows that the Newton polygon has a segment of slope $1/(m+2)$. Therefore  $F_{n, \beta} \tensor \Q_p$ also has a totally ramified component of degree $m+2$. This is a contradiction, since $F_{n, \alpha} \isom F_{n, \beta}$ are degree $n$ fields, and $2m+2>n$. This contradiction implies our claim that $\alpha,\beta\in \Z$. 

    \medskip
    To this end, we have integers $\alpha,\beta\in \Z$ such that $\alpha=\beta+2$ and $F_{n,\alpha}\cong F_{n,\beta}$ for infinitely many $n$. 
    Let 
    \[
        h(x) = (x +\alpha-2)(x+\alpha-3)=(x+\beta)(x+\beta-1)
    \]
    and for each $n>0$ let $p=p^{+}(h(n))$ (where $p^{+}(k)$ is the maximal prime divisor of $k$). By Siegel's theorem\footnote{This immediately follows from Siegel's theorem on integral points, but was established by Siegel earlier \cite{siegel1921approximation}. For a quantitative version see \cite{shorey1976greatest}.}, 
    $p\to \infty$ as $n\to \infty$. For the rest of the proof we fix a sufficiently large $n$ so that $p$ is also sufficiently large. In particular, we may assume that $h(x)$ and $h(x+2)$ are coprime modulo $p$. Since $h(n)\equiv 0\mod p$, we derive $h(n+2)\neq 0\mod p$, so $p\nmid (n+\alpha)(n+\alpha-1)$. 

    The $i$th coefficient of $L_n^{\alpha}$ is of the form $b_i=(-1)^{i} \frac{c_i}{i!(n-i)!}$, where $c_i = \prod_{j=0}^{n-i-1}(n-j+\alpha)$. Thus $p\nmid c_nc_{n-1}c_{n-2}$, so if  $p>4$, we deduce that 
    \[
        v_p(b_n)=v_p(b_{n-1})=v_p(b_{n-2}) = -v_p(n!).
    \]
    For $i\leq n-4$, we have that $h(n)\mid c_i$, so $v_p(c_i)>0$. Using that  $v_p(n!)\geq v_p(i!(n-i)!)$, we conclude that 
    \[
        v_p(b_i) = v_p(c_i) - v_p(i!(n-i)!)>-v_p(n!)=v_p(b_n).
    \]
    We conclude that the Newton polygon of $L_{n}^\alpha$ has a segment of slope $0$ and length $2$ or $3$ (depending on whether $v_p(b_{n-3})$ equals $-v_p(n!)$ or not). Thus, 
    \[
        F_{n,\beta}\tensor \Q_p\cong F_{n,\alpha}\tensor \Q_p
    \]
    has a component of degree $2$ or $3$. In particular, $L_{n}^{\beta}$ has a divisor of degree $2$ or  $3$ over $\Q_p$, so it has at least two roots with valuation in $1/2\Z \cup 1/3\Z$. 
    
    On the other hand, similar computation as before shows that the last interval in the Newton polygon of $L^{\beta}_n$ either has slope  $0$ and length $1$ or has a positive slope; hence, all roots but at most one have positive valuation. Since we have at least two roots with valuation in $1/2\Z \cup 1/3\Z$, we deduce that the maximal valuation of roots is $\geq \frac{1}{3}$. Since the first interval corresponds to the roots with the largest valuation, we deduce that the first interval has slope $\geq \frac{1}{3}$. Writing $a_i$ for the coefficients of $L_{n}^\beta$, we have $v_p(a_0)-v_p(a_k)\geq k/3$. On the other hand, a direct computation gives that 
    \[
        \frac{k}{3}\leq v_p(a_0/a_k) = v_p\left(\binom{n}{k}\prod_{j=n-k}^{n-1}(n+\beta-j)\right) \leq v_p((k+\beta)!) \leq \frac{k+\beta}{p-1},
    \]
    where the last inequality follows from Legendre's formula. Taking $p$ sufficiently large leads to a contradiction, hence finishes the proof.
\end{proof}

With the basic Galois theory of Laguerre polynomials established, we can now prove the main result of this section.

\begin{theorem}
    Given a finite collection of rational numbers $\alpha_1, \dots, \alpha_m \in \Q$, let $\calS$ be the collection of all roots of the generalized Laguerre polynomials $L_n^{\alpha_i}$ for all $n$. Then the set $\calS$ satisfies \MMgeneric{}, and so for any subset $B \subset \calS^n \subset \Aff^n(\Qbar)$, the Zariski closure $\bar{B}$ of $B$ in $\Aff^n_{\C}$ is a finite union of special subvarieties.
\end{theorem}
\begin{proof}
    We first note that by part $(3)$ of Proposition \ref{prop:basics-laguerre}, we can assume that none of the $\alpha_i$ are negative integers, as long as we add $0$ to the set $\calS$. Conditions $(1)$, and $(2)$ of property  \MMgeneric{} are verified for $\calS$ in Proposition \ref{prop:basics-laguerre}, while the disjointness condition is the subject of Lemma \ref{lem:Laguerre-disjoin}. Therefore, we can apply the generic Manin--Mumford theorem (Theorem \ref{thm:GMM-intro}) to the set $\calS$ as claimed.
\end{proof}

\section{Typical sequences of polynomials}\label{sec:RMM}
    This section is devoted to the proof of Theorem~\ref{thm:RandomMM}. Throughout, let \(H_0\ge 2\) be a sufficiently large absolute constant, let \(I\subseteq\Z\) be an interval of length \(H\ge H_0\), and let
    \[
    f_1,f_2,\ldots
    \]
    be independent random polynomials, where, for each \(n\ge 1\), \(f_n\) is chosen uniformly from the set of monic polynomials of degree \(n\) with coefficients in \(I\) and nonzero constant term. Finally, let
    \[
    \calS=\bigcup_{n=1}^{\infty}\{\alpha\in\Qbar:f_n(\alpha)=0\}\subseteq\Qbar
    \]
    denote the random set of all roots of the polynomials.

    Theorem~\ref{thm:GMM-intro} reduces the proof to the statement that \(S\) satisfies \MMgeneric{} almost surely. 

    Our approach is not specific to the uniform-on-an-interval-coefficient model considered in
    Theorem~\ref{thm:RandomMM}. The same strategy is expected to apply to other models
    of random polynomials for which sufficiently strong irreducibility and Galois-group
    estimates are available, including independent coefficients with general distributions
    \cite{breuillard2019irreducibility,BarySoroker-Koukoulopoulos-Kozma2023}
    and characteristic polynomials of random matrices    \cite{barysoroker2026tridiagonal,eberhard2022characteristic}.
    For simplicity, we restrict ourselves to the model above.

    The remainder of the section is devoted to verifying \MMgeneric{}.

\subsection{A Borel--Cantelli criterion}
    A key probabilistic tool is the
    Borel--Cantelli lemma. It is not enough for the relevant probabilities to tend to \(1\) as
    \(n\to\infty\). Instead, we require summable error terms, so that only
    finitely many exceptional values of \(n\) occur almost surely.
    
    Before formulating the criterion, we discuss a minor complication arising
    from low-degree cyclotomic factors. Let
    \[
    \Phi=(X-1)(X+1)(X^2+X+1)(X^2+1)(X^2-X+1),
    \]
    the product of all cyclotomic polynomials of degree at most \(2\). The probability that \(f_n\) is divisible by a
    factor of \(\Phi\) is  not summable. For example,
    \[
    \Prob(X-1\mid f_n)\asymp n^{-1/2}.
    \]
    Thus, by the Borel--Cantelli lemma, such factors occur infinitely often
    almost surely.
    
    However, this does not obstruct \MMgeneric{}. Indeed, the roots of \(\Phi\) form a fixed finite subset of \(\Qbar\), and
    adding or removing finitely many elements does not affect \MMgeneric{}. We may therefore ignore these roots throughout the proof. The
    following notation will be used throughout the remainder of the section.
    
    \begin{definition}\label{def:neg_cyc}
    For each \(n\), write
    \[
    f_n=\Psi_n g_n,
    \qquad
    \Psi_n=\gcd(f_n,\Phi^n),
    \]
    so that \(g_n\) is coprime to \(\Phi\). We further write
    \[
    d_n=\deg g_n,
    \]
    and let \(L_n\) denote the splitting field of \(g_n\) over \(\Q\).
    \end{definition}
    
    \begin{proposition}\label{prop:BC}
    Assume that there exists a constant \(C\) such that
    \begin{gather}
    \sum_{n=1}^{\infty}
    \Prob(\deg \Psi_n>C)
    <\infty, \label{eq:BC1}
    \\
    \sum_{n=1}^{\infty}
    \Prob\Bigl(
    \Gal(L_n/\Q)\notin
    \{A_{d_n},S_{d_n}\}
    \Bigr)
    <\infty, \label{eq:BC2}
    \\
    \sum_{n=1}^{\infty}
    \Prob\Bigl(
    \exists\,m>n \text{ such that }
    [L_n\cap L_m:\Q]>2
    \Bigr)
    <\infty. \label{eq:BC3}
    \end{gather}
    
    Then \(S\) satisfies \MMgeneric{} almost surely.
    \end{proposition}
    \begin{proof}
        By the first Borel--Cantelli lemma, almost surely only finitely many of the
        exceptional events in \eqref{eq:BC1}--\eqref{eq:BC3} occur. Fix such a
        realization.
        
        From \eqref{eq:BC1} we have \(d_n=n-\deg\Psi_n\ge n-C\) for all sufficiently
        large \(n\). Moreover, by \eqref{eq:BC2}, all but finitely many \(g_n\) have
        Galois group \(A_{d_n}\) or \(S_{d_n}\), and hence are irreducible. Thus
        every root of \(g_n\) has degree \(d_n\), and since \(d_n\to\infty\), the
        growing degree condition holds for the roots of the polynomials \(g_n\).
        
        The typical Galois group condition follows from \eqref{eq:BC2}, and the
        disjointness condition follows from \eqref{eq:BC3}. Therefore the set of
        roots of the polynomials \(g_n\) satisfies \MMgeneric{}.
        
        Finally, the roots of the factors \(\Psi_n\) belong to the fixed finite set
        of roots of \(\Phi\). Adding finitely many elements does not affect
        \MMgeneric{}, so \(S\) satisfies \MMgeneric{} as well.
    \end{proof}

    Throughout the remainder of this section, let
    \[
    p_1<\cdots<p_r,\qquad P=p_1\cdots p_r,
    \]
    where \(p_1,\ldots,p_r\) denote the first \(r\) prime numbers and \(r\) is a fixed sufficiently large integer. We allow the constant \(H_0\) to depend on \(r\). This dependence is only used in the proof of Lemma~\ref{lem:Delta_P_uniform_interval}. Since \(r\) is fixed throughout, \(H_0\) remains an absolute constant.
    
    We now verify the hypotheses of Proposition~\ref{prop:BC} for the random polynomial model of Theorem~\ref{thm:RandomMM}.

\subsection{Growing degree}\label{sec:irr}
    We show that the cyclotomic factor \(\Psi_n\) has bounded
    degree with summable exceptional probability, and that the complementary
    factor \(g_n\) is irreducible with summable exceptional probability.
    \begin{lemma}\label{lem:bounded-cyclotomic}
    We have
    \[
    \Prob(\deg \Psi_n>8)\ll n^{-3/2}.
    \]
    \end{lemma}
    
    \begin{proof}
    If \(\deg\Psi_n>8=\deg\Phi\), then some root \(\alpha\) of \(\Phi\) is a
    multiple root of \(f_n\), and hence \(f_n'(\alpha)=0\). Therefore
    \[
    \Prob(\deg\Psi_n>8)
    \le
    \Prob(f_n'(\alpha)=0 \text{ for some root } \alpha \text{ of }\Phi).
    \]
    Peled, Sen, and Zeitouni \cite[Lemma~1.4]{PeledSenZeitouni2016} proved that
    the latter probability is \(O(n^{-3/2})\) in the case where the coefficients
    of \(f_n\) are in \(\{-1,0,1\}\). The proof in \emph{loc.\ cit.} works for
    any uniform measure on an interval whose length is at least \(2\).
    \end{proof}

    \begin{lemma}\label{lem:gn-irreducible}
    Assume $r\geq 24$. Then 
    \[
    \Prob(g_n \text{ is reducible})\ll n^{-3/2}.
    \]
    \end{lemma}

    \begin{proof}
    Since \(r\ge 24\), we have
    \[
        r\left(1-\frac{1+\log\log 2}{\log 2}\right)\ge \frac32,
    \]
    and the large-factor estimate in the proof of
    \cite[Theorem~1.1.4]{Bazin25} gives
    \[
        \rho_3:=
        \Prob\bigl(
        f_n \text{ has a factor of degree in }[n^{1/10},n/2]
        \bigr)
        \ll n^{-3/2}.
    \]
    
    By
    \cite[Proposition~1.2.2]{Bazin25},
    \[
    \rho_1:=
    \Prob\bigl(
    f_n \text{ has a non-cyclotomic factor of degree }\le n^{1/10}
    \bigr)
    \ll n^{-3/2}
    \]
    and by \cite[Lemma~45]{breuillard2019irreducibility},
    \[
    \rho_2:=
    \Prob\bigl(
    \Phi_d\mid f_n
    \text{ for some }3\le \varphi(d)\le n^{1/10}
    \bigr)
    \ll n^{-3/2}.
    \]
    
    If \(g_n\) is reducible, then, since \(g_n\) has no cyclotomic factor of
    degree at most \(2\), one of the three events defining
    \(\rho_1,\rho_2,\rho_3\) must occur. Hence
    \[
    \Prob(g_n \text{ is reducible})
    \le \rho_1+\rho_2+\rho_3
    \ll n^{-3/2},
    \]
    as claimed.
    \end{proof}
    
\subsection{Typical Galois group}\label{sec:TGG}
Recall that \(r\) is fixed and sufficiently large.
\begin{proposition}
\label{prop:lgg}
    We have
    \[
    \Prob\!\left(
    \Gal(L_n/\Q)\notin\{A_{d_n},S_{d_n}\}
    \right)
    \ll n^{-3/2}.
    \]
\end{proposition}

The rest of  Section~\ref{sec:TGG} is devoted to the proof. 
We show that, with summable failure probability, one of the reductions \(g_{n,p_i}\) exhibits a partition pattern forcing every transitive subgroup containing the corresponding lift of Frobenius to be either \(A_{d_n}\) or \(S_{d_n}\). The proof follows the strategy of \cite{BarySoroker-Koukoulopoulos-Kozma2023}: describe the relevant partition events, show that they occur modulo at least one auxiliary prime, transfer them to actual cycle structures in $G_{g_n}$, and then apply a subgroup classification argument.
The key new ingredient is that we work simultaneously modulo
\(r\) auxiliary primes rather than one prime. This strengthens
the final error term obtained in
\cite{BarySoroker-Koukoulopoulos-Kozma2023}
from a power saving to a summable bound.

\subsubsection{Factorization Statistics}
We recall the relevant notation and results from \cite{BarySoroker-Koukoulopoulos-Kozma2023}, formulating them in a way suited to our application to Galois groups.

\medskip

\paragraph*{\scshape Comparison with the uniform model.}
The proof proceeds by comparing the factorization statistics of $f_n$ modulo the primes $p_1,\ldots, p_r$ with those of  independent uniformly distributed polynomials over
\(\F_{p_1},\ldots,\F_{p_r}\). To quantify this comparison, we use the
discrepancy measure introduced in
\cite[Eq.~2.6]{BarySoroker-Koukoulopoulos-Kozma2023}.

Let
\[
    \F_P[x]=\F_{p_1}[x]\times\cdots\times\F_{p_r}[x] 
\]
and 
\[
    f_{n,P}=(f_{n,p_1},\ldots,f_{n,p_r}).
\]
For \(\bfC=(C_1,\ldots,C_r)\) and \(\bfD=(D_1,\ldots,D_r)\) in $\F_P[x]$, define
\[
f_{n,P}\equiv \bfC \bmod \bfD
\qquad\Longleftrightarrow\qquad
f_{n,p_i}\equiv C_i \bmod D_i
\quad(i\in[r]).
\]
Similarly,
\[
    \bfC\mid \bfD \qquad \text{if} \qquad C_i\mid D_i \quad\text{for all }i\in[r]
\]
in which case $\bfD/\bfC=(D_i/C_i)_{i\in [r]}$. The greatest common divisor and the least common  multiple are also defined coordinate-wise:
\[
    (\bfC,\bfD) = (\gcd(C_1,D_1),\ldots, \gcd(C_r,D_r)),\qquad [\bfC,\bfD ] = \frac{\bfC\bfD}{(\bfC,\bfD)} .
\]
Let
\[
\|\bfD\|= \prod_{i\in [r]}\|D_i\| = \prod_{i\in[r]} p_i^{\deg D_i}.
\]
Then, for $Y\geq 1$, define
\begin{equation}
\label{eq:DeltaQ}
\Delta_P(n;Y)
=
\sum_{\substack{
\bfD=(D_1,\ldots,D_r)\\
\deg D_i\le Y\\
x\nmid D_i\ \forall i
}}
\max_{\bfC\bmod\bfD}
\left|
\Prob(f_{n,P}\equiv\bfC\bmod\bfD)
-\frac1{\|\bfD\|}
\right|.
\end{equation}
Given a collection $\calD\subseteq \F_P[x]$ we define the harmonic weight
\[
    \frak H(\calD) = \sum_{\bfD\in \calD} \frac{1}{\|\bfD\|}.
\]
Given collections
\[
\calI_{p_i}\subseteq \F_{p_i}[x]
\qquad (i\in[r]),
\]
consisting of monic irreducible polynomials, we write
\[
\calI=(\calI_{p_1},\ldots,\calI_{p_r})
\]
and define the associated sieve factor by
\[
\frak S(\calI)
=
\prod_{i=1}^{r}
\prod_{I_i\in\calI_{p_i}}
(1-\|I_i\|^{-1}).
\]

Writing $\bfM =\left(\prod_{I_1\in \calI_{p_1}}I_1,\ldots, \prod_{I_r\in \calI_{p_r}}I_r\right)$, we write
\[
(\bfD,\calI)
= (\bfD,\bfM), \qquad \bfD \mid \calI \Longleftrightarrow \bfD\mid \bfM.
\]
The following congruence event is instrumental in our approach:
\begin{equation}
    \label{eq:sieve_event}
    \Omega_{\bfD,\calI}= \{ f_{n,P}\equiv 0\bmod \bfD ,\ (f_{n,P}/\bfD,\calI)=1\}.
\end{equation}

For $m\ge1$, let
\[
\calI_p(m)
=
\{\,I\in\F_p[x]:
I \text{ monic irreducible},\ I\neq x,\ \deg I\le m\,\}.
\]
More generally, for
\(
\bfm=(m_1,\ldots,m_r),
\)
we write
\[
\calI_P(\bfm)
=
\bigl(
\calI_{p_1}(m_1),
\ldots,
\calI_{p_r}(m_r)
\bigr).
\]
For $D\in \F_p[x]$ of degree $n$, we let 
\begin{equation}
    \label{eq:defsmoothpart}
    D^{(m)} = \Big(D,\prod_{I\in \calI_p(m)}I^{n} \Big).
\end{equation}
So $D^{(m)}$ is the \emph{$m$-smooth part} of $D$ (for technical reasons, we deviate mildly from the canonical notation by excluding the powers of the prime $I=x$ from the smooth part).
We will use  the following version of the Prime Polynomial Theorem. Denote by $\pi_p(j)$ the number of irreducible monic polynomials of degree $j$ in $\F_p[x]$. Then, 
\begin{equation}
    \label{eq:PPT}
    \frac{1}{j} -\frac{2p^{-j/2}}{j} \leq \frac{\pi_p(j)}{p^{j}}\leq \frac{1}{j},
\end{equation}
\cite[Proposition 2.1]{Rosen2002}. We also use the  following analogs of Mertens' theorems. For any $p$ and $m$
\[
    \frak{H}_p(m):= \frak{H}(\calI_p(m)) = \sum_{k=1}^m \frac{1}{k} + O(1)=\log m+O(1),
\]
with absolute implied constant and 
\[
        \frak{S}_p(m) :=\frak{S}(\calI_p(m)) \leq \frac{2}{m+1},
\]
The first is immediate from \eqref{eq:PPT} and the second is an exercise, see for example \cite[Lemma~8.3]{BarySoroker-Koukoulopoulos-Kozma2023}.

\begin{definition}\label{def:controllable}
A \emph{covering datum} $(\calD_a,\calI_a)_{a\in A}$ consists of a finite set $A$ and for each $a\in A$ 
\begin{enumerate}
    \item 
    a collection $\calD_a$ of tuples
    \[
        \bfD=(D_1,\ldots, D_r)\in \F_P[x]
    \]
    with $x\nmid D_i$ for all $i\in [r]$,
    \item 
    $\calI_a = (\calI_{p_1,a},\ldots, \calI_{p_r,a})$ with $\calI_{p_i,a}$ a collection of irreducible monic polynomials.
\end{enumerate}
We say that an event $\AA$ is \emph{covered} by $(\calD_a,\calI_a)_{a\in A}$ if
\[
    \AA\subseteq \Union_{a\in A}\Union_{\bfD\in\calD_a} \Omega_{\bfD,\calI_a},
\]
where $\Omega_{\bfD,\calI_a}$ is as defined in  \eqref{eq:sieve_event}.

For $N,M,R,S,T\in \R_{\geq 0}$, we say that $\AA$ is \emph{$(N,M,R,S,T)$-controllable} if it is covered by $(\calD_a,\calI_a)_{a\in A}$ and 
\begin{enumerate}
    \item $\deg D_{i}\leq N$ for all $a\in A$, $i\in [r]$, and $\bfD\in \calD_{a}$,
    \item $\deg I_i\leq M$ for all $a\in A$, $i\in [r]$, and $I_i\in \calI_{p_i,a}$,
    \item if $R_a$ denotes the maximal number of representations of $\bfC\in \F_P[x]$ of the form $\bfC=\bfD\bfG$ with $\bfD\in\calD_a$, $\bfG\mid \calI_a$, and each $G_i$ has at most $6\log M_a$ irreducible factors, then $\sum_{a\in A}R_a \leq R$,
    \item $\sum_{a\in A} \frak{S}(\calI_a)\frak{H}(\calD_a)\leq S$,
    \item $|A|\leq T$.    
\end{enumerate}

    If $\calI_a=\emptyset$ for all $a\in A$, then we may take $M=0$ and $R=0$, so we simplify and say that $\AA$ is \emph{$(N,S,T)$-controllable} and that it is covered by $(\calD_{a})_{a\in A}$. 

    We will often consider events that only depend on $f_{n,p}$ for a single prime $p\mid P$. In this case, it suffices to find collections of $\calD_{a}\subseteq \F_{p}[x]$ and of monic irreducibles $\calI_a$. Indeed, we identify $\F_p[x]$ with $\ker(\F_P[x]\to \F_{P/p}[x])$, that is $D\in \calD_a$ is identified with the tuple whose other coordinates are $1$ and we set $\calI_{q,a}=\emptyset$ for $q\mid P/p$.
\end{definition}

\begin{lemma}\label{lem:union_controllable}
    Let $B$ be a finite set, and for each $b\in B$, let  $\AA_b$ be an $(N_b,M_b,R_b,S_b,T_{b})$-controllable event. Then 
    $\AA:=\Union_{b\in B}\AA_b$ is $(N,M,R,S,T)$-controllable, where 
\[
\begin{aligned}
N &= \max_{b\in B} N_b,
&\qquad
M &= \max_{b\in B} M_b,\\
R &= \sum_{b\in B} R_b,
&\qquad
S &= \sum_{b\in B} S_b,
&\qquad
T &= \sum_{b\in B} T_b.
\end{aligned}
\]
\end{lemma}

\begin{proof}
    In the covering data for $\AA_b$, we may assume without loss of generality that the index sets $A_b$ are disjoint, so let $A= \coprod_{b\in B} A_b$ be their disjoint union. Then $(\calD_a,\calI_a)_{a\in A}$ is the requested covering datum for $\AA$.
\end{proof}

\begin{lemma}
    \label{lem:compare-uniform}
    Let $\AA$ be an $(N,M,R,S,T)$-controllable event with $M\geq 11$. Then
    \[
        \Prob(\AA) \leq 2^r S + R \Delta_P(n;N+6M\log M).
    \]
    If $M=0$, that is, if $\AA$ is $(N,S,T)$-controllable, then
    \[
        \Prob(\AA)\leq S + T\Delta_P(n;N).
    \]
\end{lemma}

\begin{proof}
    First assume that \(\AA\) is \((N,S,T)\)-controllable. Then
    \[
        \AA\subseteq
        \bigcup_{a\in A}\bigcup_{\bfD\in\calD_a}
        \{f_{n,P}\equiv0\bmod \bfD\}.
    \]
    Hence, by the union bound and the definition of \(\Delta_P(n;N)\), for each
    \(a\in A\),
    \[
    \begin{aligned}
    \sum_{\bfD\in\calD_a}
    \Prob(f_{n,P}\equiv0\bmod\bfD)
    &\le
    \sum_{\bfD\in\calD_a}\frac1{\|\bfD\|}
    +
    \sum_{\bfD\in\calD_a}
    \left|
    \Prob(f_{n,P}\equiv0\bmod\bfD)-\frac1{\|\bfD\|}
    \right|  \\
    &\le
    \frak H(\calD_a)+\Delta_P(n;N).
    \end{aligned}
    \]
    Summing over \(a\in A\), and using \(|A|\le T\), completes the proof. 
    
    In the general case, we apply the sieve bound \cite[Lemma 8.2]{BarySoroker-Koukoulopoulos-Kozma2023}, writing $X=N+6M\log M$ and using the notation as in Definition~\ref{def:controllable}, we get 
    \[
        \Prob(\AA) \leq \sum_{a\in A} (2^{r} \frak{S}(\calI_a)\frak{H}(\calD_a) + R_a \Delta_P(n;X))\leq 2^{r} S+R\Delta_P(n;X),
    \]
    as claimed. 
\end{proof}

Lemma~\ref{lem:compare-uniform} is only useful when \(\Delta_P(n;X)\) is sufficiently small. The following lemma shows that this is indeed the case, provided \(H_0\) is chosen sufficiently large. This is the only place in the proof where the assumption that \(H_0\) is sufficiently large is used.

\begin{lemma}\label{lem:Delta_P_uniform_interval}
There exists \(H_0=H_0(r)\) such that, if \(H\ge H_0\), then
\begin{equation}
    \label{eq:Delta}
    \Delta_P(n;X)
    \ll e^{-n^{1/10}},
\end{equation}
uniformly for \(X\le \frac n2+n^{0.88}\).
\end{lemma}

\begin{proof}
Let \(\mu\) be the uniform measure on \(I\). By the argument in the proof of
\cite[Theorem~1(a)]{BarySoroker-Koukoulopoulos-Kozma2023}, if \(H_0\) is
chosen sufficiently large in terms of \(P\), then
\[
    \max_{\substack{Q_1R=P\\ Q_1>1}}
    \max_{\ell\in\Z}
    \frac{1}{\sqrt{Q_1}}
    \sum_{k\in\Z/Q_1\Z}
    \left|
        \widehat\mu\!\left(\frac{k}{Q_1}+\frac{\ell}{R}\right)
    \right|
    <1.
\]
Thus the hypotheses of
\cite[Proposition~2.3]{BarySoroker-Koukoulopoulos-Kozma2023}
hold with \(s=1\) and \(\gamma=1/2\), and the bound
\eqref{eq:Delta} follows.
\end{proof}

\medskip
\paragraph*{\scshape Multiplicity bounds and mergings.}
We shall need two elementary notions concerning multiplicities modulo the
auxiliary primes. First, for a prime \(p\) and an integer \(\ell\ge 1\), consider the following  event on $f_n$ defined by 
\begin{equation}\label{eq:highmul}
\MM(p,\ell)
=
\left\{
\exists I\in\F_q[x] \text{ irreducible with } I^\ell\mid f_{n,p}
\right\}.
\end{equation}

\begin{lemma}\label{lem:SmallMult}
    Let $\ell = \lceil \log n\rceil$. Then, there exists $c>0$ such that 
    \[
        \Prob(\exists i\in [r] : f_{n,p_i}\in \MM(p_i,\ell^3))\ll n^{-c\log n}.
    \]
\end{lemma}

\begin{proof}
    By the union bound, it suffices to prove that for each fixed $p=p_i$,
    \[
        \Prob(f_{n,p}\in \MM(p,\ell^3))\ll n^{-c\log n}.
    \]
    We first separate the possible contribution of the irreducible polynomial $x$.
    For every fixed prime $p$, there is a constant $\rho_p<1$ such that, for any
    interval of integers of length $H\geq 2$, the probability that a uniformly
    chosen element of the interval is divisible by $p$ is at most $\rho_p$. Removing
    $0$ from the interval, if this occurs for the constant term, can only decrease
    this probability. Hence the same $\rho_p$ applies to all relevant coefficients.
    By independence,
    \[
        \Prob(x^{\ell^3}\mid f_{n,p})
        \leq \rho_p^{\ell^3} = n^{-c_p(\log n)^2}.
    \]

    If there exists an irreducible $I\neq x$ with $I^{\ell^3}\mid f_{n,p}$, then $\deg I\leq n/\ell^3$, and in particular 
    \[
        \deg I^{\ell^2}\leq \frac{n}{\ell}\leq \frac n2
    \]
    for $n$ sufficiently large. Thus, if 
    \[
        \calD_p = \{ D^{\ell^2}: x\nmid D, \ 0<\deg D\leq n/\ell^{3}\},
    \]
    then 
    $\calD_p$ covers our event, so it is  $(n/\ell,S,1)$-controllable with 
    \[
        S\leq  \sum_{\deg D>0}\|D\|^{-\ell^2} \leq \sum_{j=1}^\infty p^{j(1-\ell^{2})}\ll p^{1-\ell^2}=n^{-c\log n}.
    \]
    Now Lemma~\ref{lem:compare-uniform} and \eqref{eq:Delta} completes the proof. 
\end{proof}

A consequence of Lemma~\ref{lem:SmallMult} is that the factorization patterns of the reductions \(f_{n,p_i}\) can be transferred to cycle structures occurring in the Galois group \(G_{f_n}\). 

More precisely, let \(\calP_n\) denote the set of partitions of \(n\), viewed as multisets.
For \(\sigma\in S_n\), let
\[
\rho_\sigma\in\calP_n
\]
be the cycle structure of \(\sigma\). 
Likewise, for a polynomial
\(h\in\F_p[X]\) of degree \(n\), let
\[
\rho_h\in\calP_n
\]
be the cycle structure of \(h\), namely the partition whose parts are the
degrees of the irreducible factors of \(h\), counted with multiplicity. 
Finally, the reduction of $f_n$ modulo the primes $p_1,\ldots, p_r$ defines an $r$-tuple of random partitions, namely  for each  \(p\), define
\begin{equation}
    \label{eq:randompartrho_i}
    \rho_p=\rho_{f_{n,p}}\in\calP_n.
\end{equation}

Recall that, roughly speaking, a \(y\)-merging of a partition \(\rho\) is
obtained from \(\rho\) by allowing groups of at most \(y\) parts of the same
size to merge into larger parts; see
\cite[Definition~11.2]{BarySoroker-Koukoulopoulos-Kozma2023} for the precise
definition.

\begin{corollary}\label{cor:merging}
With probability \(1-O(n^{-c\log n})\), 
for each $p\mid P$, there exists
\(\sigma\in G_{f_n}\) such that
\(\rho_{\sigma}\) is a \((\log n)^3\)-merging of \(\rho_{p}\).
\end{corollary}

\begin{proof}
By Lemma~\ref{lem:SmallMult}, with probability \(1-O(n^{-c\log n})\), we have
\[
f_{n,p}\notin \MM(p,\ell^3), \qquad \ell=\lceil\log n\rceil,
\]
for every \(p\mid P\). The conclusion therefore follows from
\cite[Proposition~11.3]{BarySoroker-Koukoulopoulos-Kozma2023}.
\end{proof}

\medskip
\paragraph*{\scshape Auxiliary factorization events.}
We begin with an event controlling the size of the \(m\)-smooth part of
\(f_{n,p}\): For $p\mid P$, $m\in [n]$ and $u\geq 1$, let 
\[
    \KK_1(p,m,u) = \{ \deg f_{n,p}^{(m)} >um\}
\]
\begin{lemma}\label{lem:bigsmoothpart}
    For $C>1$ fixed, $n\geq 3$, $p\mid P$, $m\in [4C,n]$, and $u\geq2$, $\KK_1(p,m,u)$
    is $(um+m,S,1)$-controllable with $S\ll_C m e^{-Cu}$.
\end{lemma}

\begin{proof}
    In the proof of \cite[Lemma~9.1]{BarySoroker-Koukoulopoulos-Kozma2023} it is shown that  
    \[
        \calD = \{D=D^{(m)},\ x\nmid D,\ um<\deg D\le um+m\}\subseteq \F_p[x]
    \]
    is a covering of $\KK_1(p,m,u)$ and that $\frak{H}(\calD)\ll me^{-Cu}$. This completes the proof.  
\end{proof}

A function \(\phi\colon \F_p[x]\smallsetminus\{0\}\to \R\) is called \emph{additive} if 
\[
\phi(D_1D_2)=\phi(D_1)+\phi(D_2)
\]
whenever \((D_1,D_2)=1\).

Two additive functions that will play a prominent role are
\[
\omega(D)
=
\sum_{\substack{I\mid D\\ I\ \mathrm{irreducible}}}1,
\qquad
\log_2\tau(D)
=
\frac{\log \tau(D)}{\log 2},
\]
where
$
\tau(D)
=
\sum_{G\mid D}1.
$
Both also satisfy \eqref{eq:phi-cond} below.

For an arithmetic function \(\phi\) and $m\geq 1$, define
\[
L_{\phi}(m)
=
\sum_{I\in\calI_p(m)}
\phi(I)\|I\|^{-1}.
\]

The next lemma concerns the values of an arithmetic function on the
$m$-smooth part of a polynomial. For $t\geq 0$, let
\begin{align*}
        \KK_{2}^{+}(p,\phi,t,m)
        &=
        \Bigl\{
        \phi(f_{n,p}^{(m)})
        \ge tL_{\phi}(m)
        \Bigr\},\\
        \KK_{2}^{-}(p,\phi,t,m)
        &=
        \Bigl\{
        \phi(f_{n,p}^{(m)})
        \le tL_{\phi}(m)
        \Bigr\},\\
        \KK_2(p,\phi,t,m) &=  \KK_2^{\pm_t}(p,\phi,t,m),
\end{align*}
where $\pm_t =\begin{cases}
    +, & t>1\\
    -,& t<1.
\end{cases}$
\begin{lemma}\label{lem:typical_addfun}
    Let \(C_1,C_2\geq 3\), let \(p\mid P\), let $t\in [0,C_2]$, and let \(\phi\) be an additive
    function on \(\F_p[x]\) such that
    \begin{equation}
        \label{eq:phi-cond}
        \phi(I)\in\{0,1\},
        \qquad
        0\leq \phi(I^\nu)\leq C_1\log \nu
    \end{equation}
    for every irreducible \(I\) and every \(\nu\geq2\). 
 
    Then there exists an absolute constant \(C_3\) such that, if
    \(C_3\le m\le n/\log n\), the event $\KK_2(p,\phi,t,m)$ is 
    \[
        (n/4+\log n,\ m/24,\ R,\ S,\ m+2)\text{-controllable},
    \]
    where
    \[
        R\ll_{C_2} n^{C_2+5},
    \]
    and
    \[
        S\ll_{C_1,C_2}
    mn^{-C/4}
    +\exp\bigl(-(t\log t-t+1)L_\phi(m)\bigr)
    +m^{-(t_0\log t_0-t_0+1)},
\]
with
\[
    t_0=
    \begin{cases}
        3,& 0\le t\le 1,\\
        \max(t,2),& 1\le t\le C_2.
    \end{cases}
\]
\end{lemma}

\begin{proof}
    We first consider the special case \(\phi=\omega\) and
    \(t_0\in[2,C_2]\). For this case, we use the decomposition in
    \cite[Lemma~9.2]{BarySoroker-Koukoulopoulos-Kozma2023}, with
    \(\theta=1/2\), \(A_p\) playing the role of \(f_{n,p}\), and the
    additive function \(f\) in \emph{loc.\ cit.} playing the role of
    \(\omega\).

    Let
    \[
        \KK_1=\KK_1\left(p,m,\frac{\log n}{4}\right).
    \]
    For \(j\in[m]\), set
    \[
        \KK_{2,j,t_0}
        =
        \left\{
            \omega(f_{n,p}^{(j)})
            \ge t_0L_{\omega}(m)
            >
            \omega(f_{n,p}^{(j-1)})
        \right\}\smallsetminus \KK_1 .
    \]
    Between Equations (9.6) and (9.7) of \emph{loc.\ cit.}, it is shown
    that \(\KK_{2,j,t_0}\) is covered by
    \((\mathcal D_j,\mathcal I_p(y_j))\), where
    \[
        y_j=\max\left(11,\left\lfloor\frac{j-1}{24}\right\rfloor\right)
    \]
    and
    \[
        \mathcal D_j=
        \left\{
        BJ:
        \begin{array}{l}
        B=B^{(j)},\ J \text{ irreducible},\ \deg J=j,\ \deg(BJ)\le n/4,\\
        t_0L_\omega(m)-1\le \omega(B)<t_0L_\omega(m)
        \end{array}
        \right\}.
    \]
    Moreover, the displayed equation immediately preceding Equation~(9.8)
    in \emph{loc.\ cit.} gives
    \[
        R_j\ll m^{t_0+4}\le n^{C_2+5},
    \]
    where \(R_j\) is the representation multiplicity appearing in
    Definition~\ref{def:controllable}. The paragraph following
    Equation~(9.8) gives
    \[
        \sum_{j=1}^{m}
        \mathfrak S(\mathcal I_p(y_j))\mathfrak H(\mathcal D_j)
        \ll_{C_2} m^{-(t_0\log t_0-t_0+1)}.
    \]
    Hence
    \[
        \widetilde{\KK}_{2,t_0}
        :=
        \KK_2^{+}(p,\omega,t_0,m)\smallsetminus \KK_1
        \subseteq \bigcup_{j\le m}\KK_{2,j,t_0}
    \]
    is
    \[
        \left(
            \frac n4,\,
            \frac{m}{24},\,
            O_{C_2}(n^{C_2+5}),\,
            O_{C_2}\bigl(m^{-(t_0\log t_0-t_0+1)}\bigr),\,
            m
        \right)\text{-controllable}.
    \]

    We now return to a general additive function \(\phi\) satisfying
    \eqref{eq:phi-cond}. The final part of the proof of
    \cite[Lemma~9.2]{BarySoroker-Koukoulopoulos-Kozma2023} shows that, for
    any set \(X\subseteq\mathbb R_{\ge0}\), the event
    \[
        \{\phi(f_{n,p}^{(m)})\in X\}
        \setminus
        \bigl(\KK_1\cup\widetilde{\KK}_{2,t_0}\bigr)
    \]
    is covered by \((\mathcal D,\mathcal I_p(y_m))\), where
    \[
        \mathcal D
        =
        \left\{
            B:
            B=B^{(m)},\ \deg B\le n/4,\ \phi(B)\in X
        \right\}.
    \]
    The corresponding representation multiplicity satisfies
    \[
        R\le n^{t_0+4}\le n^{C_2+5}.
    \]

    For the upper-tail event \(\KK_2^+(p,\phi,t,m)\), where
    \(t\in[1,C_2]\), we take
    \[
        X=[tL_\phi(m),\infty),
        \qquad
        t_0=\max(t,2).
    \]
    For the lower-tail event \(\KK_2^-(p,\phi,t,m)\), where
    \(t\in[0,1]\), we take
    \[
        X=[0,tL_\phi(m)],
        \qquad
        t_0=3.
    \]
    In both cases, the same proof in \emph{loc.\ cit.} gives
    \[
        \mathfrak S(\mathcal I_p(y_m))\mathfrak H(\mathcal D)
        \ll_{C_1,C_2}
        \exp\bigl(-(t\log t-t+1)L_\phi(m)\bigr).
    \]
    Therefore the corresponding main event is
    \[
        \left(
            \frac n4,\,
            \frac{m}{24},\,
            O_{C_2}(n^{C_2+5}),\,
            O_{C_1,C_2}\!\left(
                \exp\bigl(-(t\log t-t+1)L_\phi(m)\bigr)
            \right),\,
            1
        \right)\text{-controllable}.
    \]

    Finally, by Lemma~\ref{lem:bigsmoothpart}, the event \(\KK_1\) is
    \[
        \left(
            \frac n4+\log n,\,
            S_1,\,
            1
        \right)\text{-controllable},
        \qquad
        S_1\ll_C mn^{-C/4}.
    \]
    Applying Lemma~\ref{lem:union_controllable} to \(\KK_1\),
    \(\widetilde{\KK}_{2,t_0}\), and the main event associated with \(X\),
    we conclude that both \(\KK_2^+(p,\phi,t,m)\) and
    \(\KK_2^-(p,\phi,t,m)\) are
    \[
        (n/4+\log n,\ \frac{m}{24},\ R,\ S,\ m+2)\text{-controllable},
    \]
    where
    \[
        R\ll_{C_2} n^{C_2+5}
    \]
    and
    \[
        S\ll_{C_1,C_2}
        mn^{-C/4}
        +\exp\bigl(-(t\log t-t+1)L_\phi(m)\bigr)
        +m^{-(t_0\log t_0-t_0+1)}.
    \]
    This proves the lemma.
\end{proof}

Our final auxiliary event, and the main one for our purposes, concerns
the existence of a factor of prescribed degree. For $k\le n$, let
\[
    \KK_3(p,k)
    =
    \{\exists D\mid f_{n,p}:\ \deg D=k\}.
\]
Since a divisor of degree $k$ corresponds to a complementary divisor of
degree $n-k$, we have
\[
    \KK_3(p,k)=\KK_3(p,n-k).
\]
Accordingly, it suffices to consider $k\le n/2$.
\begin{lemma}\label{lem:div_deg_k}
Let \(p\mid P\), \(1\le k\le n/2\), and let
\[
    \MM=\MM(p,\lceil\log n\rceil^3)
\]
be as in \eqref{eq:highmul}. Let \(\epsilon,\lambda\in(0,1/4)\).
Then, for \(k\) sufficiently large in terms of \(\epsilon,\lambda\), there exists $c_\epsilon>0$ depending only on $\epsilon$ such that the event
\[
    \KK_3(p,k)\smallsetminus\MM
\]
is \((N,M,R,S,T)\)-controllable, where
\[
\begin{array}{l@{\hspace{5em}}l}
    N=\dfrac n2+\lambda\epsilon k^\lambda,
    &
    M=\dfrac{n}{24\log n},
    
    \\[1ex]
    R,T\ll_{\epsilon,\lambda} n^{O(1)},
    &
    S\ll_{\epsilon,\lambda} k^{-\lambda c_\epsilon/10}.
\end{array}
\]
\end{lemma}

\begin{proof}
Let
\[
    \KK_3'(p,k)
    =
    \{\exists D\mid f_{n,p}:\ \deg D=k,\ x\nmid D\}.
\]
Since \(f_{n,p}\notin\MM\) implies that the multiplicity of \(x\) in
\(f_{n,p}\) is at most \((\log n)^3\), we have
\[
    \KK_3(p,k)\smallsetminus\MM
    \subseteq
    \bigcup_{0\le \nu\le (\log n)^3}
    \KK_3'(p,k-\nu).
\]
Thus, by Lemma~\ref{lem:union_controllable} it suffices to prove that \(\KK_3'(p,k)\) is
\[
    \left(N,M,\frac{R}{(\log n)^3},
    \frac{S}{(\log n)^3},n\right)\text{-controllable}.
\]

We want to work outside the exceptional event
\[
    \UU_k
    =
    \bigcup_{k^{\lambda/2}\le m\le n/\log n}
        \left(
            \KK_1(p,m,\epsilon\log m)
            \cup
            \KK_2^{+}(p,\log_2\tau,t_m,m)
        \right),
\]
where
\[
    t_m
    =
    \left(1+\frac{\epsilon}{3}\right)
    \frac{\log m}{L_{\log_2\tau}(m)} = 1+\frac{\epsilon}{3} + o_{m\to \infty}(1).
\]
It is $(N_{\UU}, M_{\UU},R_{\UU},S_{\UU},T_{\UU})$-controllable by Lemmas~\ref{lem:union_controllable}, \ref{lem:bigsmoothpart}, and
\ref{lem:typical_addfun},  where
\[
    N_\UU=\frac n4+\log n,
    \qquad
    M_\UU=\frac{n}{24\log n}, \qquad R_{\UU},T_{\UU}\ll_{\epsilon} n^{O(1)}, \qquad S_{\UU}\ll k^{-\lambda c_{\epsilon}/2}.
\]
We elaborate further on the latter bound:
Lemma~\ref{lem:bigsmoothpart}, applied with
\(u=\epsilon\log m\), gives
\[
    S_{\KK_1(p,m,\epsilon\log m)}
    \ll m e^{-C\epsilon\log m}
    =
    m^{1-C\epsilon}.
\]
On the other hand, Lemma~\ref{lem:typical_addfun}, applied to
\(\phi=\log_2\tau\), gives
\[
    S_{\KK_2^+(p,\log_2\tau,t_m,m)}
    \ll_\epsilon
    m n^{-C/4}
    +
    \exp\{-(t_m\log t_m-t_m+1)L_{\log_2\tau}(m)\}
    +
    m^{-c_\epsilon}.
\]
Since \(L_{\log_2\tau}(m)=\log m+O(1)\), the exponential term is
\[
    \ll_\epsilon m^{-(t_m\log t_m-t_m+1)}.
\]
Also \(t_m=1+\epsilon/3+o(1)\), so for all sufficiently large \(m\),
\[
    t_m\log t_m-t_m+1\ge c_\epsilon>0.
\]
Choosing the constant \(C\) in Lemma~\ref{lem:bigsmoothpart}
sufficiently large in terms of \(\epsilon\), and using
\(m\ge k^{\lambda/2}\), we get
\[
    S_\UU
    \ll_\epsilon
    \sum_{k^{\lambda/2}\le m\le n/\log n}
    \left(
        m^{1-C\epsilon}
        +
        m^{-c_\epsilon}
    \right)
    \ll_{\epsilon,\lambda} k^{-\lambda c_\epsilon/2}.
\]

It remains to prove the required controllability for
\(\KK_3''(p,k):=\KK_3'(p,k)\smallsetminus\UU_k\).

    In the proof of
    \cite[Lemma~9.4]{BarySoroker-Koukoulopoulos-Kozma2023}, with
    \(\theta=1/2\) and \(r=1\) in \emph{loc.\ cit.}  it is shown that $\KK_3''(p,k)$ 
    is covered by \((\mathcal D_{k},\mathcal I_p(m_k))\), where $m_k = \lfloor\frac{k^{\lambda}}{8\log n}\rfloor$ and 
    \[
        \mathcal D_{k} =
        \left\{
        [B,D] :
        \begin{array}{l}
        B^{(m_k)}=B,\quad \deg B\le \epsilon m_k\log m_k,\quad \deg D=k,\\
        \tau(B)\le m_k^{(1+\epsilon)\log 2},\quad D^{(m_k)}\mid B,\quad x\nmid D
        \end{array}
        \right\}.
    \]
    The corresponding representation multiplicity is at most \(m_k^6\), and
    \[
        \mathfrak S(\mathcal I_p(m_k))\mathfrak H(\mathcal D_{k})
        \ll
        k^{-\lambda(1-\log 2-\epsilon)}\log n,
    \]
    for \(k\) sufficiently large in terms of \(\epsilon,\lambda\).  Here we use again that 
        $L_{\log_2\tau}(m)
        =
        \log m+O(1),$
    so that \(t_m\in(1,2)\) for large \(m\).
    This completes the proof.
\end{proof}

\subsubsection{Partition Events Excluding Proper Transitive Subgroups}
    Following \cite[Section~12.1]{BarySoroker-Koukoulopoulos-Kozma2023}, we introduce partition events that rule out proper transitive subgroups. More precisely, if the cycle structure of a permutation \(\sigma\) is a sufficiently small merging of a partition satisfying these events, then every transitive subgroup containing \(\sigma\) is either \(A_n\) or \(S_n\).
    
    Our goal is to show that, with a summable error term, at least one of the random partitions
    \[
    \rho_{p},\qquad p\mid P
    \]
    defined in \eqref{eq:randompartrho_i} satisfies all of these events. 

    If \(\EE\subseteq\calP_n\) is a subset of partitions, then for each
    prime \(p\mid P\) we denote by
    \[
    \EE(p)
    =
    \{\rho_{p}\in\EE\}
    \]
    the corresponding event on $f_n$.

    For $\alpha,\kappa>0$, define
    \[
        \EE_1=\EE_{1,\alpha,\kappa}
        =
        \left\{
            \rho\in\calP_n:
            \forall k,\ell\le \frac n4,\ 
            \gcd(k,\ell)\ge n^{\kappa\alpha}
            \Longrightarrow
            \{k,\ell\}\not\subseteq \rho
        \right\}.
    \]
    
    \begin{lemma}\label{lem:E1-controllable}
    For $\alpha,\kappa>0$, for \(p\mid P\), and all sufficiently large \(n\), the event \(\overline{\EE_1}(p)\) is
    \((n/2,S_1,1)\)-controllable, where
    \[
        S_1\ll (\log n)^2 n^{-\kappa\alpha}.
    \]
    \end{lemma}   
    
    \begin{proof}
    Let \(\calD\) be the collection of polynomials \(D=IJ\), where
    \(I,J\in\F_p[x]\) are monic irreducible polynomials of degrees \(k,\ell\le n/4\),
    respectively, and
    \[
        \gcd(k,\ell)\ge n^{\kappa\alpha}.
    \]
    Then \(\deg D\le n/2\), and for \(n\) sufficiently large we also have
    \(x\nmid D\).
    
    If \(f_n\in\overline{\EE_1}(p)\), then \(\rho_p\) has two parts
    \(k,\ell\le n/4\) with \(\gcd(k,\ell)\ge n^{\kappa\alpha}\). Hence
    \(f_{n,p}\) has irreducible factors \(I,J\) of degrees \(k,\ell\), counted with
    multiplicity, and therefore \(IJ\mid f_{n,p}\) for some \(IJ\in\calD\).
    Thus \(\calD\) covers \(\overline{\EE_1}(p)\).
    
    It remains to estimate the harmonic weight. Grouping the degrees \(k,\ell\)
    according to \(d=\gcd(k,\ell)\), and using that $\pi_{p}(a)/p^{a}\leq 1/a$, gives
    \[
    \begin{aligned}
        S_1:=\frak H(\calD)
        &\le
        \sum_{d\ge n^{\kappa\alpha}}
        \sum_{a,b\le n/(4d)}
        \frac{\pi_p(da)}{p^{da}}
        \frac{\pi_p(db)}{p^{db}}  \\
        &\leq
        \sum_{d\ge n^{\kappa\alpha}}
        \frac1{d^2}
        \sum_{a,b\le n/(4d)}
        \frac1{ab}
        \ll
        (\log n)^2 n^{-\kappa\alpha},
    \end{aligned}
    \]
    as needed. 
    \end{proof}

    For $\delta > 0$, define 
    \[
        \EE_2 = \EE_{2,\delta}
        =
        \left\{
        \rho\in\calP_n:
        \forall a\mid n,\ 2\le a\le n^{\delta/2},\
        \forall j\in[a-1],\
        \frac{nj}{a}
        \text{ cannot be expressed as a sum of parts of }\rho
        \right\}.
    \]

    \begin{lemma}\label{lem:E2-controllable}
    Let \(p\mid P\), and put
    \[
        \MM_p=\MM(p,\lceil\log n\rceil^3),
    \]
    as in \eqref{eq:highmul}. Let \(\epsilon,\lambda\in(0,1/4)\) and $\delta < \frac{\lambda c_{\epsilon}}{20}$, where \(c_{\epsilon}>0\) is the constant of
    Lemma~\ref{lem:div_deg_k}. For sufficiently large \(n\) in terms of
    \(\epsilon,\lambda,\delta\), the event 
    \(\overline{\EE_2}(p)\smallsetminus\MM_p\) is
    \((N_2,M_2,R_2,S_2,T_2)\)-controllable, where
    \[
        N_2=\frac n2+\lambda\epsilon n^\lambda,
        \qquad
        M_2=\frac{n}{24\log n},
        \qquad
        R_2,T_2\ll_{\epsilon,\lambda,\delta} n^{O(1)},
        \qquad
        S_2\ll_{\epsilon,\lambda,\delta}
        n^{-\delta}.
    \]
    \end{lemma}
        
    \begin{proof}
    Let
    \[
        \Omega
        =
        \left\{
            (a,j):
            a\mid n,\ 2\le a\le n^{\delta/2},\ 1\le j\le \frac a2
        \right\}.
    \]
    Assume \(\overline{\EE_2}(p)\). Then there exist
    \(a\mid n\), \(2\le a\le n^{\delta/2}\), \(j\in[a-1]\) such that $\frac{nj}{a}$
    can be expressed as a sum of parts of \(\rho_p\).
    Replacing \(j\) by \(a-j\) if necessary, we may assume that
    \(j\le a/2\). Since sums of parts of \(\rho_p\) are in bijection with
    degrees of divisors of \(f_{n,p}\), we conclude that
    \(f_{n,p}\) has a divisor of degree \(nj/a\). Hence
    \[
        \overline{\EE_2}(p)\smallsetminus\MM_p
        \subseteq
        \bigcup_{(a,j)\in\Omega}
        \left(
            \KK_3\left(p,\frac{nj}{a}\right)\smallsetminus\MM_p
        \right).
    \]
    For \((a,j)\in\Omega\), put \(k=nj/a\). Then
    \[
        n^{1-\delta/2}\le k\le \frac n2,
    \]
    so \(k\) is sufficiently large for Lemma~\ref{lem:div_deg_k}. Applying that
    lemma to each \(k=nj/a\), and then using Lemma~\ref{lem:union_controllable},
    gives the claimed values of \(N_2,M_2,R_2,T_2\), except for the estimate on
    \(S_2\).
    
    It remains to bound the total \(S\)-contribution. By Lemma~\ref{lem:div_deg_k},
    with
    \[
        \beta=\frac{\lambda c_\epsilon}{10},
    \]
    the contribution of \(k=nj/a\) is
    \[
        \ll k^{-\beta}
        =
        \left(\frac{nj}{a}\right)^{-\beta}.
    \]
    Therefore, 
    \[
        S_2
        \ll
        \sum_{(a,j)\in \Omega}
        \left(\frac{nj}{a}\right)^{-\beta}  
        \ll
        n^{-\beta}
        \sum_{a\le n^{\delta/2}}
        a^\beta
        \sum_{1\le j\le a/2} j^{-\beta}
        \ll
        n^{-\beta}
        \sum_{a\le n^{\delta/2}} a
        \ll
        n^{-\beta+\delta}\leq n^{-\delta},
    \]
    since $\beta>2\delta$.
    This proves the asserted bound for \(S_2\), and completes the proof.
    \end{proof}
    
    For $\alpha,t\in(0,1)$ define
    \[
        \EE_3 = \EE_{3,\alpha,t}
        =
        \left\{
            \rho\in\calP_n:
            \rho
            \text{ has at least }
            \frac{\alpha t}{2}\log n
            \text{ parts in }
            [n^{1-\alpha},n/\log n]
        \right\}.
    \]
    
    \begin{lemma}\label{lem:E3-controllable}
    For $t,\alpha\in (0,1)$, \(p\mid P\), and all sufficiently large
    \(n\), the event \(\overline{\EE_3}(p)\) is
    \((N_3,M_3,R_3,S_3,T_3)\)-controllable, where
    \[
        N_3=\frac n4+\log n,
        \qquad
        M_3=\frac{n}{24\log n},
        \qquad 
        R_3,T_3\ll n^{O(1)},
        \qquad
        S_3\ll n^{-c_3}
    \]
    for some constant \(c_3=c_3(\alpha,t)>0\).
    \end{lemma}
        
    \begin{proof}
    Put
    \(
        m=\left\lfloor \frac{n}{\log n}\right\rfloor.
    \)
    Define an additive function \(\phi\) on \(\F_p[x]\) by
    \[
        \phi(I^\nu)
        =
        \begin{cases}
            1, & n^{1-\alpha}\le \deg I\le m,\\
            0, & \text{otherwise},
        \end{cases}
    \]
    for every irreducible \(I\) and every \(\nu\ge1\). Then \(\phi\) satisfies
    \eqref{eq:phi-cond}. Moreover, by \eqref{eq:PPT}
    \[
        L_\phi(m)
        =
        \sum_{n^{1-\alpha}\le \deg I\le m}\|I\|^{-1}
        =
        \sum_{n^{1-\alpha}\le j\le m}\frac1j+O(1)
        =
        \alpha\log n+O(\log\log n).
    \]
    For sufficiently large \(n\), we therefore have
    \[
        \frac{\alpha t}{2}\log n
        \le
        \frac{2t}{3}L_\phi(m).
    \]
    As 
    \(
        \phi(f_{n,p}^{(m)}) 
    \)
    counts the number of distinct irreducible factors of $f_{n,p}$ with degree in $[n^{1-\alpha},m]$, and the number of parts of $\rho_p$ is the same but with multiplicity, we conclude that 
    \[
        \overline{\EE_3}(p)
        \subseteq
        \KK_2^{-}\left(p,\phi,\frac{2t}{3},m\right).
    \]
    
    Thus, by Lemma~\ref{lem:typical_addfun}, $\overline{\EE_3}(p)$ is  an \((N_3,M_3,R_3,S_3,T_3)\)-controllable event, with
    \[
        N_3=\frac n4+\log n,
        \qquad
        M_3=\frac{m}{24}\le \frac{n}{24\log n},
        \qquad
        R_3,T_3\ll n^{O(1)}.
    \]
    Finally, since \(2t/3<1\), the large-deviation term in
    Lemma~\ref{lem:typical_addfun} gives
    \[
        \exp\left(
        -\left(\frac{2t}{3}\log\frac{2t}{3}-\frac{2t}{3}+1\right)
        L_\phi(m)
        \right)
        \ll n^{-c_3}
    \]
    for some \(c_3=c_3(\alpha,t)>0\). The remaining terms in the bound for \(S_3\)
    are also \(O(n^{-c_3})\), after decreasing \(c_3\) if necessary. This proves
    the lemma.
    \end{proof}
    
    Define
    \[
        U_4
        =
        \left\{
            u\le \frac{\sqrt n}{3}:
            \exists q>n^{1/8}\text{ prime such that }q\mid u
        \right\},
    \]
    and for $t\in (0,1)$ let
    \[
        \EE_4=\EE_{4,t}
        =
        \left\{
            \rho\in\calP_n:
            \rho
            \text{ has at least }
            \frac{t\log n}{4}
            \text{ parts in }U_4
        \right\}.
    \]
    
    \begin{lemma}\label{lem:E4-controllable}
    For $t\in(0,1)$, \(p\mid P\), and all sufficiently large
    \(n\), the event \(\overline{\EE_4}(p)\) is
    \((N_4,M_4,R_4,S_4,T_4)\)-controllable, where
    \[
        N_4=\frac n4+\log n,
        \qquad
        M_4=\frac{\sqrt n}{72},
        \qquad
        R_4,T_4\ll n^{O(1)},
        \qquad
        S_4\ll n^{-c_4}
    \]
    for some constant \(c_4=c_4(t)>0\).
    \end{lemma}
    
    \begin{proof}
    We argue exactly as in the proof of
    Lemma~\ref{lem:E3-controllable}, replacing the interval
    \([n^{1-\alpha},n/\log n]\) by the set \(U_4\).
    So \(\phi\) is defined by
    \[
        \phi(I^\nu)
        =
        \begin{cases}
            1, & \deg I\in U_4,\\
            0, & \text{otherwise}.
        \end{cases}
    \]
    Moreover, by \eqref{eq:PPT},
    \[
        L_\phi(m)
        =
        \sum_{\substack{u\le \sqrt n/3\\ \exists q>n^{1/8},\ q\mid u}}
        \frac1u
        +O(1).
    \]
    The elementary estimate used in
    \cite[Lemma~12.5]{BarySoroker-Koukoulopoulos-Kozma2023} gives
    \[
        L_\phi(m)\ge \left(\frac14+o(1)\right)\log n.
    \]
    Hence, for \(n\) sufficiently large,
    \[
        \overline{\EE_4}(p)
        \subseteq
        \KK_2^{-}(p,\phi,t',m)
    \]
    for, say $t'=\frac{1+t}{2}\in (t,1)$.
    
    The conclusion now follows from
    Lemma~\ref{lem:typical_addfun} exactly as in the proof of
    Lemma~\ref{lem:E3-controllable}.
    \end{proof}

For $\alpha \in (0,1/2)$, define
\[
    \EE_5=\EE_{5,\alpha}
    =
    \left\{
        \rho\in\calP_n:
        \forall a\ge2,\ 
        \rho \text{ has a part }
        u\in[n^{1-\alpha},n/\log n]
        \text{ such that }a\nmid u
    \right\}.
\]

\begin{lemma}\label{lem:E5-controllable}
Let \(\alpha,t \in (0,1)\), and let \(\kappa>0\) be sufficiently small in terms of
\(\alpha\) and \(t\). For every prime \(p\mid P\), and all sufficiently large
\(n\), the event
\[
    \overline{\EE_{5,\alpha}}(p)
    \cap
    \EE_{1,\alpha,\kappa}(p)
    \cap
    \EE_{3,\alpha,t}(p)
\]
is \((N_5,M_5,R_5,S_5,T_5)\)-controllable, where
\[
    N_5=\frac n4+\log n,
    \qquad
    M_5=\frac{n}{24\log n},
    \qquad,
    R_5,T_5\ll n^{O(1)},
    \qquad
    S_5\ll n^{-c_5}
\]
for some constant \(c_5=c_5(\alpha,\kappa,t)>0\).
\end{lemma}

\begin{proof}
Put
\[
    m=\left\lfloor\frac{n}{\log n}\right\rfloor,
    \qquad
    J=[n^{1-\alpha},m].
\]
Assume that
\[
    \rho_p\in
    \overline{\EE_{5,\alpha}}
    \cap
    \EE_{1,\alpha,\kappa}
    \cap
    \EE_{3,\alpha,t}.
\]
Then there exists \(a\ge2\) such that every part of \(\rho_p\) lying in \(J\)
is divisible by \(a\). Since \(\rho_p\in\EE_{3,\alpha,t}\) and $n$ sufficiently large, there are two parts $u,v\in J$. Since $a\mid u$ and 
\(a\mid v\) and on $\EE_1$,
\[
    a\le \gcd(u,v)\leq n^{\kappa\alpha}.
\]
Thus it suffices to consider integers \(a\) with
\[
    2\le a<n^{\kappa\alpha}.
\]

For each such \(a\), define an additive function \(\phi_a\) by
\[
    \phi_a(I^\nu)
    =
    \begin{cases}
        1, & \deg I\in J \text{ and } a\nmid \deg I,\\
        0, & \text{otherwise}.
    \end{cases}
\]
Then,
\[
    \overline{\EE_{5,\alpha}}(p)
    \cap
    \EE_{1,\alpha,\kappa}(p)
    \cap
    \EE_{3,\alpha,t}(p)
    \subseteq
    \bigcup_{2\le a<n^{\kappa\alpha}}
    \KK_2^{-}(p,\phi_a,0,m).
\]

We apply Lemma~\ref{lem:typical_addfun} to each event on the right. For every
\(a\), the parameters are
\[
    N=\frac n4+\log n,
    \qquad
    M=\frac m{24}\le \frac{n}{24\log n},
    \qquad
    R\ll n^{O(1)},
    \qquad
    T\le m+2\ll n.
\]
It remains to estimate the total \(S\)-contribution.

By \eqref{eq:PPT}
\[
    L_{\phi_a}(m)
    =
    \sum_{\substack{u\in J\\ a\nmid u}}\frac1u+O(1).
\]
Uniformly for \(2\le a<n^{\kappa\alpha}\), we have
\[
\begin{aligned}
    \sum_{\substack{u\in J\\ a\nmid u}}\frac1u
    &=
    \sum_{n^{1-\alpha}\le u\le m}\frac1u
    -
    \sum_{\substack{n^{1-\alpha}\le u\le m\\ a\mid u}}\frac1u  \\
    &=
    \alpha\log n+O(\log\log n)
    -
    \frac1a
    \sum_{n^{1-\alpha}/a\le v\le m/a}\frac1v
    +O(1)  \\
    &\ge
    \left(1-\frac1a\right)\alpha\log n
    +O(\log\log n)
    \gg_\alpha \log n.
\end{aligned}
\]
In particular, since \(a\ge2\),
\[
    L_{\phi_a}(m)\ge \frac{\alpha}{3}\log n
\]
for all sufficiently large \(n\), uniformly in \(a\).

Lemma~\ref{lem:typical_addfun}, applied with lower-tail parameter \(0\), gives
\[
    S_{\KK_2^{-}(p,\phi_a,0,m)}
    \ll
    m n^{-C/4}
    +
    e^{-L_{\phi_a}(m)}
    +
    m^{-(3\log 3-3+1)}.
\]
Therefore, after choosing \(C\) sufficiently large,
\[
    S_{\KK_2^{-}(p,\phi_a,0,m)}
    \ll
    n^{-c'_5}
\]
for some \(c'_5=c'_5(\alpha)>0\), uniformly for
\(2\le a<n^{\kappa\alpha}\). Summing over \(a<n^{\kappa\alpha}\), and taking
\(\kappa>0\) sufficiently small in terms of \(\alpha\), gives
\[
    S_5
    \ll
    n^{\kappa\alpha} n^{-c'_5}
    \ll
    n^{-c_5}
\]
for some \(c_5=c_5(\alpha,\kappa)>0\).

Finally, Lemma~\ref{lem:union_controllable} gives the asserted
\((N_5,M_5,R_5,S_5,T_5)\)-controllability, with
\[
    N_5=\frac n4+\log n,
    \qquad
    M_5=\frac{n}{24\log n},
    \qquad
    R_5,T_5\ll n^{O(1)},
\]
as needed.
\end{proof}

We fix once and for all
\[
    \epsilon=\lambda=\frac18, \qquad \alpha=t=\frac12.
\]
Choose \(\delta>0\) sufficiently small so that Lemma~\ref{lem:E2-controllable}
applies, and \(\kappa>0\) sufficiently small so that Lemma~\ref{lem:E5-controllable} applies. With these choices all parameters are absolute.
Let $\EE = \EE_1\cap \cdots \cap \EE_{5}$.

\begin{proposition}\label{prop:partition-events}
Assume that $H_0$ is sufficiently large so that \eqref{eq:Delta} holds.  There exist absolute constants \(c>0\) and \(n_0=n_0(c,r)\ge1\) such that, for all
\(n\ge n_0\),
\[
    \Prob\bigl(\rho_p\notin\EE,\ \forall p\mid P\bigr)
    \ll_r n^{-cr}.
\]
\end{proposition}

\begin{proof}
Let
$\MM=\bigcup_{p\mid P}\MM(p,\lceil\log n\rceil^3)$.
By Lemma~\ref{lem:SmallMult},
\[
    \Prob(\MM)\ll n^{-c'\log n}\ll n^{-cr}.
\]
Hence we may assume that we are outside of $\MM$ for the rest of the proof. 

For \(p\mid P\), set
\[
\begin{aligned}
    \BB_1(p)&=\overline{\EE_{1,\alpha,\kappa}}(p),\\
    \BB_2(p)&=\overline{\EE_{2,\delta}}(p)\smallsetminus \MM(p,\lceil\log n\rceil^3),\\
    \BB_3(p)&=\overline{\EE_{3,\alpha,t}}(p),\\
    \BB_4(p)&=\overline{\EE_{4,t}}(p),\\
    \BB_5(p)&=\overline{\EE_{5,\alpha}}(p)
        \cap \EE_{1,\alpha,\kappa}(p)
        \cap \EE_{3,\alpha,t}(p).
\end{aligned}
\]
Then, outside \(\MM\), and with $\bfj=(j_1,\ldots, j_r)$, we have 
\[
    \{\rho_p\notin\EE,\ \forall p\mid P\}
    \subseteq
    \bigcup_{\bfj \in [5]^r}
    \bigcap_{i=1}^{r}\BB_{j_i}(p_i).
\]

Fix $\bfj\in [5]^r$. By Lemmas~\ref{lem:E1-controllable}--\ref{lem:E5-controllable}, each
\(\BB_{j_i}(p_i)\) is controllable with $N$-parameter $\leq \frac{n}{2}+n^{1/4}$, $M$-parameter $\leq \frac{n}{24 \log n}$, $R,T$-parameters $\ll n^{O(1)}$, and  \(S\)-parameter \(\ll n^{-c_j}\), for some
\(c_j>0\). Let
$c_0=\min_{1\le j\le5}c_j>0$.
Thus $\bigcap_{i=1}^{r}\BB_{j_i}(p_i)$
is controllable with the same $N$ and $M$-parameters and with  \(S\)-parameter \(\ll n^{-c_0r}\), while the corresponding
\(R\)- and \(T\)-parameters are \(n^{O(r)}\). Hence, $N+6M\log M\leq \frac{n}{2}+n^{0.88}$ for large $n$. Thus Lemma~\ref{lem:compare-uniform}
and \eqref{eq:Delta} give
\[
    \Prob\left(\bigcap_{i=1}^r\BB_{j_i}(p_i)\right)
    \ll
    n^{-c_0r}+n^{O(r)}e^{-n^{1/10}}
    \ll n^{-c_0r}.
\]
Summing over the \(5^r\) possible $\bfj$,  gives
\[
    \Prob(\rho_p\notin\EE,\ \forall p\mid P)
    \ll_r n^{-cr}.
\]
This proves the proposition.
\end{proof}

\subsubsection{Proof of Proposition~\ref{prop:lgg}}
Let $c$ be the constant from Proposition~\ref{prop:partition-events}, choose $r$ sufficiently large so that $cr>3/2$, and $H_0$ sufficiently large so that Lemma~\ref{lem:Delta_P_uniform_interval} holds. Let $n$ be sufficiently large. 

By Proposition~\ref{prop:partition-events}, with probability at least $1-Cn^{-cr}$, there exists $p\mid P$ such that \(\rho_p\in\EE\). By Corollary~\ref{cor:merging}, with probability at least $1-Cn^{-cr}$, there exists $\sigma\in G_{f_n}$ whose cycle structure is a $(\log n)^3$-merging of $\rho_p$. Hence, with probability $1-O(n^{-3/2})$, there exist $p\mid P$ and $\sigma\in G_{f_n}$ such that $\rho_\sigma$ is a $(\log n)^3$-merging of a partition in \(\EE\).

Then, by Lemmas~\ref{lem:bounded-cyclotomic} and \ref{lem:gn-irreducible}, there exists an irreducible factor $g_n$ of $f_n$ of degree $d_n=\deg g_n\geq n-8$ with probability at least $1-Cn^{-3/2}$. 

Restricting $\sigma$ to its action on the roots of $g_n$ yields an element
$\tau\in G_{g_n}$. Since $d_n\ge n-8$, the cycle structures
$\rho_\sigma$ and $\rho_\tau$ differ only by deleting parts of total size
at most $8$. In particular, for sufficiently large $n$, every partition
event defining \(\EE\) is preserved under this operation. Hence
$\rho_\tau$ is also a $(\log n)^3$-merging of a partition in \(\EE\).

By \cite[Lemma~12.8]{BarySoroker-Koukoulopoulos-Kozma2023},
every primitive subgroup of $S_{d_n}$ containing such an element $\tau$
is either $A_{d_n}$ or $S_{d_n}$. The proof in \emph{loc.\ cit.}\ applies
verbatim with $d_n$ in place of $n$: indeed, since $d_n\ge n-8$, the
inequality
\[
\frac{\sqrt n}{3}\le \frac{\sqrt{d_n}-1}{2}
\]
holds for all sufficiently large $n$, which is the only numerical input
required in the argument.

Since $g_n$ is irreducible, the group $G_{g_n}$ is transitive. Therefore,
if $G_{g_n}$ is primitive, we conclude that
\[
G_{g_n}\in\{A_{d_n},S_{d_n}\}.
\]

Finally, \cite[Lemma~12.9]{BarySoroker-Koukoulopoulos-Kozma2023}
shows that no imprimitive subgroup can contain an element whose cycle
structure is a $(\log n)^3$-merging of a partition in \(\EE\). Thus
$G_{g_n}$ cannot be imprimitive, and therefore
\[
G_{g_n}\in\{A_{d_n},S_{d_n}\}.
\]
This completes the proof.
\qed

\subsection{Disjointness}
We keep the notation of the previous sections. For each $n\geq 2$, let $d_n=\deg g_n$ and $L_n$ the splitting field of $g_n$ over $\Q$. We showed that with probability at least $1-Cn^{-3/2}$ we have 
\begin{equation}
    \label{eq:previous_sec}
    n-8\leq d_n\leq n \qquad \text{and}\qquad \Gal(L_n/\Q)=S_{d_n} \text{ or } A_{d_n}.
\end{equation}
\begin{proposition}\label{prop:disjointness}
    With probability at least $1-O(n^{-3/2})$ we have
    \[
        [L_n\cap L_m:\Q]\leq 2,
    \]
    for all $m> n\geq 15$.
\end{proposition}

The idea is that, once the Galois groups are large, a non-quadratic
intersection forces the splitting fields to coincide. We then show that
two independent random polynomials are very unlikely to have isomorphic
permutation representations, yielding the required summable bound.

\begin{lemma}\label{lem:specificm}
    With probability at least $1-Cn^{-3/2}$, if there exists $m>n$ such that $[L_n\cap L_{m}]>2$, then $m\leq n+8$, $d_n=d_m$, $\Q[x]/g_n\cong \Q[x]/g_m$, and  $L_{n} = L_m$. 
\end{lemma}
\begin{proof}
    With probability \(1-O(n^{-3/2})\), we have
    \[
        \Gal(L_n/\Q)\in\{A_{d_n},S_{d_n}\}
    \]
    and, for every \(m>n\), the same holds for \(\Gal(L_m/\Q)\).
    We may assume that \(n\ge 13\), so that \(d_n\ge 5\).

    The rest of the proof is deterministic. Suppose that \([L_n\cap L_m:\Q]>2\). Since \(L_n\cap L_m\) is
    Galois over \(\Q\), the group
    \[
        \Gal(L_n\cap L_m/\Q)
    \]
    is a common quotient of \(\Gal(L_n/\Q)\) and \(\Gal(L_m/\Q)\).
    For \(d\ge 5\), the only proper nontrivial quotient of \(S_d\) is
    the quadratic quotient, while \(A_d\) has no nontrivial proper
    quotient. Therefore a common quotient of degree \(>2\) forces the
    full alternating/symmetric quotient to occur on both sides. Hence
    \(d_n=d_m\) and
    \[
        L_n\cap L_m=L_n=L_m .
    \]

    Finally, \(A_d\) and \(S_d\) each have a unique conjugacy class of
    subgroups of index \(d\). Since \(g_n\) and \(g_m\) are irreducible
    of the same degree \(d=d_n=d_m\) and have the same splitting field,
    the corresponding degree-\(d\) subfields are conjugate. Thus
    \[
        \Q[x]/g_n\cong \Q[x]/g_m .
    \]
\end{proof}

\begin{proof}[Proof of Proposition~\ref{prop:disjointness}]
    For $m>n$ let $\BB_{n,m}$ be the event that \(d_n=d_m, \ \Q[x]/g_n \cong \Q[x]/g_m\). For $n$ sufficiently large, by Lemma~\ref{lem:specificm}
    \[
        \Prob\left(\exists m>n \ [L_n:L_m]>2 \right)\leq \sum_{m=n+1}^{n+8} \Prob\left(\BB_{n,m}\right) + O(n^{-3/2})
    \]
    Hence, it suffices to show that for a fixed $m\in [n+1,n+8]$, the event $\BB_{n,m}$ occurs with probability $\ll n^{-3/2}$. 
    
    On $\BB_{n,m}$, the actions of $G=\Gal(L_n/\Q)=\Gal(L_m/\Q)$ on the set of roots of $g_n$ and on the set of roots of $g_m$ are isomorphic. Hence, the cycle structure of $\sigma\in G$ with respect to the action of any of these polynomials is the same, and we shall denote it by $\rho_{\sigma}$.
    
    Put \(\ell=\lceil\log n\rceil^3\). Outside the event
    \[
        \MM:=\MM_P(n,\ell)\cup \MM_P(m,\ell),
    \]
    both $g_{n,p}$ and $g_{m,p}$ are $\ell$-free. Thus, letting $\sigma\in G$ be a lifting of the Frobenius modulo $p$, the proof of  \cite[Lemma~11.3]{BarySoroker-Koukoulopoulos-Kozma2023} gives that 
    \[
        \rho_{\sigma}
    \]
    is an $\ell$-merging of $\rho_{g_{n,p}}$ and of $\rho_{g_{m,p}}$.
    
    Consider \(\EE(P)\), the event that \(f_{n,p}\in \EE\) (with
    \(\alpha=1/5\)) for some \(p\mid P\). Then, by
    Lemma~\ref{lem:SmallMult} and
    Proposition~\ref{prop:partition-events},
    \[
        \Prob(\BB_{n,m})
        \leq
        \Prob(\BB(n,m)\cap \EE(P)\smallsetminus \MM)
        +O(n^{-3/2}).
    \]
    Hence it suffices to prove that
    \[
        \Prob(\BB')\ll n^{-3/2},
        \qquad
        \BB'=\BB(n,m)\cap \EE(P)\smallsetminus \MM .
    \]

    On \(\EE(P)\), there exists \(p\mid P\) such that
    \(\rho_{f_{n,p}}\in\EE\). Since \(\alpha=1/5\), the event
    \(\EE_3\) yields distinct parts
    \[
        k_1,k_2\in [n^{4/5},n/\log n]
    \]
    of \(\rho_{f_{n,p}}\). Moreover, since
    \(\rho_{f_{n,p}}\in\EE_1\), each of \(k_1\) and \(k_2\) occurs
    with multiplicity one. Since \(\rho_{g_{n,p}}\) is obtained from
    \(\rho_{f_{n,p}}\) by deleting parts of total size at most \(8\),
    we still have
    \[
        \{k_1,k_2\}\subseteq \rho_{g_{n,p}}
    \]
    for \(n\) sufficiently large.

    Since $\rho$ is an $\ell$-merging of $\rho_{g_{n,p}}$ and $\ell<n^{4/5}\leq k_1,k_2$ for large $n$, we conclude that $\{k_1,k_2\}\subseteq \rho$. Since $\rho$ is an $\ell$-merging of $\rho_{g_{m,p}}$, there exists $1\leq \ell_i\leq \ell$ and $k_i'$ such that $k_i=\ell_i k_i'$ and $\{k_1',k_2'\}\subseteq \rho_{g_{m,p}}\subseteq \rho_{f_{m,p}}$. Thus, on $\BB'$, there exist irreducible polynomials $I_1,I_2$ 
    \begin{equation}
    \label{eq:degrees}        
        I_1I_2\mid f_{m,p}, \qquad \ell_i\deg I_i=k_i,
        \qquad i=1,2.
    \end{equation}
    
    Thus, conditioning $f_{n}$ as above, taking
    \(
    \{k_1,k_2\}\subseteq \rho_{f_{n,p}}
    \)
    with \(k_1,k_2\) as above, and using that \(f_n\) and \(f_m\) are sampled
    independently, the event \(\BB'\) becomes an event depending only on
    \(f_m\). 
    Hence, viewed as an event on \(f_m\), it is covered by
    \[
        \calD
        =
        \bigcup_{\substack{\ell_1\mid k_1, \\ \ell_1\leq \ell}}
        \bigcup_{\substack{\ell_2\mid k_2, \\ \ell_2\leq \ell}}
        \{
            I_1I_2 :
            I_1,I_2 \text{ irreducible},
            \ \ell_i\deg I_i=k_i,\ i=1,2
        \}.
    \]
    Thus, if  $I_1I_2\in\calD$, there exists $\ell_i\mid k_i$ with $\ell_i\leq \ell$. Then,
    \[
        k_i/\ell_{i}\in \left[\frac{n^{4/5}}{\ell}, \frac{ n}{\log n}\right] \subseteq [\ell n^{3/4}, n/4],
    \]
    for $n$ sufficiently large. By \eqref{eq:PPT}, for $k\in [\ell n^{3/4}, n/4]$ we have $\sum_{\deg I=k}\|I\|^{-1}\leq \frac{1}{k}\leq \ell^{-1}n^{-3/2}$. Thus,
    \[
        \frak{H}(\calD) = \sum_{\substack{\ell_1\mid k_1, \\ \ell_1\leq \ell}}
        \sum_{\substack{\ell_2\mid k_2, \\ \ell_2\leq \ell}} \sum_{\substack{\deg I_i=k_i/\ell_i, \\ i=1,2}} \|I_1\|^{-1}\|I_2\|^{-1} 
        \leq 
        \sum_{\substack{\ell_1\mid k_1, \\ \ell_1\leq \ell}}
        \sum_{\substack{\ell_2\mid k_2, \\ \ell_2\leq \ell}} \frak{S}_{k_1/\ell_1}\frak{S}_{k_2/\ell_2}
        \leq \ell^{2}\ell^{-2}n^{-3/2}\leq    n^{-3/2}.
    \]
    So $\BB'$ conditioned on $\{\{k_1,k_2\}\subseteq \rho_{f_{n,p}}\}$ is $(n/2,n^{-3/2},1)$-controllable event on $f_m$. By  Lemma~\ref{lem:compare-uniform} and \eqref{eq:Delta},
    \[
    \Prob(\BB' | \{k_1,k_2\}\subseteq \rho_{f_{n,p}})\ll n^{-3/2}.
    \]
    Finally, applying law of total probability, we conclude that 
    \begin{multline}
        \Prob(\BB')  
        = \sum_{\substack{k_1, k_2\in [n^{4/5},n/\log n]\\ k_1\neq k_2}} \Prob(\BB' | \{k_1,k_2\}\subseteq  \rho_{f_{n,p}}) \Prob(\{k_1,k_2\}) \subseteq \rho_{f_{n,p}}) 
        \\
        \ll n^{-3/2} \sum_{\substack{k_1, k_2\in [n^{4/5},n/\log n]\\ k_1< k_2}}\Prob(\{k_1,k_2\}) \subseteq \rho_{f_{n,p}})\leq n^{-3/2},
    \end{multline}
    as needed to complete the proof. 
\end{proof}

\subsection{Proof of Theorem~\ref{thm:RandomMM}}
    Since $\sum_{n=1}^{\infty} n^{-3/2} <\infty$, the conditions (1)--(3) of  Proposition~\ref{prop:BC} hold (with $C=8$) by Lemma~\ref{lem:bounded-cyclotomic},  Proposition~\ref{prop:lgg}, and Proposition~\ref{prop:disjointness}. Hence, 
    $S$ satisfies \MMgeneric{}, and the proof is done by Theorem~\ref{thm:GMM-intro}.
\qed

\bibliographystyle{alpha}
\bibliography{bibliography}
\end{document}